\documentclass[fleqn,reqno,10pt,a4paper,final]{amsart}

\usepackage[a4paper,left=35mm,right=35mm,top=30mm,bottom=30mm,marginpar=25mm]{geometry} 
\usepackage{amsmath}
\usepackage{amssymb}
\usepackage{amsthm}
\usepackage{amscd}
\usepackage[ansinew]{inputenc}
\usepackage{cite}
\usepackage{bbm}
\usepackage{color}
\usepackage[english=american]{csquotes}
\usepackage[final]{graphicx}
\usepackage{hyperref}

\graphicspath{{Pictures/}}

\numberwithin{equation}{section}

\newtheoremstyle{thmlemcorr}{10pt}{10pt}{\itshape}{}{\bfseries}{.}{10pt}{{\thmname{#1}\thmnumber{ #2}\thmnote{ (#3)}}}
\newtheoremstyle{thmlemcorr*}{10pt}{10pt}{\itshape}{}{\bfseries}{.}\newline{{\thmname{#1}\thmnumber{ #2}\thmnote{ (#3)}}}
\newtheoremstyle{defi}{10pt}{10pt}{\itshape}{}{\bfseries}{.}{10pt}{{\thmname{#1}\thmnumber{ #2}\thmnote{ (#3)}}}
\newtheoremstyle{remexample}{10pt}{10pt}{}{}{\bfseries}{.}{10pt}{{\thmname{#1}\thmnumber{ #2}\thmnote{ (#3)}}}
\newtheoremstyle{ass}{10pt}{10pt}{}{}{\bfseries}{.}{10pt}{{\thmname{#1}\thmnumber{ A#2}\thmnote{ (#3)}}}

\theoremstyle{thmlemcorr}
\newtheorem{theorem}{Theorem}
\numberwithin{theorem}{section}
\newtheorem{lemma}[theorem]{Lemma}
\newtheorem{corollary}[theorem]{Corollary}
\newtheorem{proposition}[theorem]{Proposition}

\theoremstyle{thmlemcorr*}
\newtheorem{theorem*}{Theorem}
\newtheorem{lemma*}[theorem]{Lemma}
\newtheorem{corollary*}[theorem]{Corollary}
\newtheorem{proposition*}[theorem]{Proposition}
\newtheorem{problem*}[theorem]{Problem}
\newtheorem{conjecture*}[theorem]{Conjecture}

\theoremstyle{defi}
\newtheorem{definition}[theorem]{Definition}

\theoremstyle{remexample}
\newtheorem{remark}[theorem]{Remark}

\theoremstyle{ass}

\newcommand{\Crm}{\mathrm{C}}

\newcommand{\Lrm}{\mathrm{L}}

\newcommand{\Wrm}{\mathrm{W}}

\newcommand{\Gcal}{\mathcal{G}}
\newcommand{\Hcal}{\mathcal{H}}

\newcommand{\Lcal}{\mathcal{L}}

\newcommand{\Rcal}{\mathcal{R}}
\newcommand{\Scal}{\mathcal{S}}

\newcommand{\Ebf}{\mathbf{E}}

\newcommand{\Mbf}{\mathbf{M}}

\newcommand{\Ybf}{\mathbf{Y}}

\renewcommand{\Bbb}{\mathbb{B}}

\newcommand{\Sbb}{\mathbb{S}}

\DeclareMathOperator{\id}{id}

\DeclareMathOperator*{\wslim}{w*-lim}

\DeclareMathOperator{\supmod}{sup}

\DeclareMathOperator{\rank}{rank}

\DeclareMathOperator{\supp}{supp}

\newcommand{\setn}[2]{\{\, #1 \ \ \textup{\textbf{:}}\ \ #2 \,\}}
\newcommand{\setb}[2]{\bigl\{\, #1 \ \ \textup{\textbf{:}}\ \ #2 \,\bigr\}}

\newcommand{\setBB}[2]{\biggl\{\, #1 \ \ \textup{\textbf{:}}\ \ #2 \,\biggr\}}

\newcommand{\norm}[1]{\|#1\|}

\newcommand{\normn}[1]{\|#1\|}
\newcommand{\normb}[1]{\bigl\|#1\bigr\|}

\newcommand{\abs}[1]{|#1|}

\newcommand{\absn}[1]{|#1|}
\newcommand{\absb}[1]{\bigl|#1\bigr|}
\newcommand{\absB}[1]{\Bigl|#1\Bigr|}
\newcommand{\absBB}[1]{\biggl|#1\biggr|}

\newcommand{\dpr}[1]{\langle #1 \rangle}	

\newcommand{\dprn}[1]{\langle #1 \rangle}
\newcommand{\dprb}[1]{\bigl\langle #1 \bigr\rangle}

\newcommand{\ddpr}[1]{\langle\!\langle #1 \rangle\!\rangle}

\newcommand{\ddprn}[1]{\langle\!\langle #1 \rangle\!\rangle}
\newcommand{\ddprb}[1]{\bigl\langle\hspace{-2.5pt}\bigl\langle #1 \bigr\rangle\hspace{-2.5pt}\bigr\rangle}
\newcommand{\ddprB}[1]{\Bigl\langle\!\!\Bigl\langle #1 \Bigr\rangle\!\!\Bigr\rangle}

\newcommand{\cl}[1]{\overline{#1}}
\newcommand{\di}{\mathrm{d}}
\newcommand{\dd}{\;\mathrm{d}}

\newcommand{\N}{\mathbb{N}}
\newcommand{\R}{\mathbb{R}}

\newcommand{\loc}{\mathrm{loc}}

\newcommand{\reg}{\mathrm{reg}}
\newcommand{\sing}{\mathrm{sing}}
\newcommand{\ONE}{\mathbbm{1}}

\newcommand{\toweakstar}{\overset{*}\rightharpoondown}

\newcommand{\todown}{\downarrow}

\newcommand{\SmallO}{\mathrm{\textup{o}}}
\newcommand{\sbullet}{\begin{picture}(1,1)(-0.5,-2)\circle*{2}\end{picture}}
\newcommand{\frarg}{\,\sbullet\,}
\newcommand{\BV}{\mathrm{BV}}

\newcommand{\GY}{\mathbf{GY}}

\newcommand{\toY}{\overset{\Ybf}{\to}}

\newcommand{\eps}{\epsilon}

\DeclareMathOperator{\Tan}{Tan}

\newcommand{\term}[1]{\textbf{#1}}

\newcommand{\proofstep}[1]{\textit{#1}}

\def\Xint#1{\mathchoice 
{\XXint\displaystyle\textstyle{#1}}%
{\XXint\textstyle\scriptstyle{#1}}%
{\XXint\scriptstyle\scriptscriptstyle{#1}}%
{\XXint\scriptscriptstyle\scriptscriptstyle{#1}}%
\!\int} 
\def\XXint#1#2#3{{\setbox0=\hbox{$#1{#2#3}{\int}$} 
\vcenter{\hbox{$#2#3$}}\kern-.5\wd0}} 
\def\dashint{\,\Xint-}

\newcommand{\restrict}{\begin{picture}(10,8)\put(2,0){\line(0,1){7}}\put(1.8,0){\line(1,0){7}}\end{picture}}

\renewcommand{\eps}{\varepsilon}
\renewcommand{\phi}{\varphi}

\title[Young measure characterization in BV]{A local proof for the characterization of Young measures generated by sequences in BV}

\author{Filip Rindler}
\address{Mathematics Institute, University of Warwick, Coventry CV4 7AL, United Kingdom, and, United Kingdom.}
\email{F.Rindler@warwick.ac.uk}

\begin{document}


\begin{abstract}
This work presents a new proof of the recent characterization theorem for generalized Young measures generated by sequences in BV by Kristensen and the author [\textit{Arch.\ Ration.\ Mech.\ Anal.\ 197 (2010), 539--598}]. The present argument is based on a localization technique together with a local Hahn--Banach argument in novel function spaces combined with an application of Alberti's Rank-One Theorem. This strategy avoids employing a relaxation theorem as in the previously known proof, and the new tools introduced in its course should prove useful in other contexts as well. In particular, we introduce \enquote{homogeneous} Young measures, separately at regular and singular points, which exhibit rather different behaviour than the classical homogeneous Young measures. As an application, we show how for BV-Young measures with an \enquote{atomic} part one can find a generating sequence respecting this structure.

\vspace{4pt}



\vspace{4pt}

\noindent\textsc{Date:} \today{} (version 3.0).
\end{abstract}

\hypersetup{
  pdfauthor = {Filip Rindler (University of Warwick)},
  pdftitle = {A local proof for the characterization of Young measures generated by sequences in BV},
  pdfsubject = {MSC (2010): 49J45 (primary); 28B05, 46G10},
  pdfkeywords = {Young measure characterization, BV, quasiconvexity}
}

\maketitle


\section{Introduction}

The characterization of a class of Young measures generated by constrained sequences is a recurring problem in Young measure theory with many applications in the analysis of nonlinear PDEs and the Calculus of Variations. The first major results in this field are due to Kinderlehrer \& Pedregal~\cite{KinPed91CYMG,KinPed94GYMG,Pedr97PMVP} and concern sequences of gradients bounded in $\Wrm^{1,p}$ with $1 < p \leq \infty$ (and also for $p = 1$ if one additionally assumes equiintegrability). Their characterization puts the generated \enquote{gradient} Young measures in duality with quasiconvex functions with $p$-growth (quasiconvex functions were introduced by Morrey~\cite{Morr52QSMI}, for a modern introduction see Dacorogna's book~\cite{Daco08DMCV}). Some ideas also go back to Tartar's lecture notes~\cite{Tart79CCAP} and the investigations into microstructure by Ball \& James~\cite{BalJam87FPMM}. For further related results see~\cite{KruRou96ECLP,Kris99LSSW,Sych00CHGY,FonKru10OCGA,SzeWie12YMGI}.

Part of the interest in Young measure characterization theorems in the spirit of Kinderlehrer \& Pedregal stems from the fact that they reveal a duality to the (generalized) convexity class relevant in the corresponding minimization problems. Indeed, the \enquote{necessity} direction of such a characterization result is equivalent to (sequential) weak lower semicontinuity of integral functionals with integrands in the corresponding convexity class. On the other hand, the \enquote{sufficiency} direction is interesting for instance in relaxation theory, cf.~\cite{Pedr97PMVP}: If a given variational problem has no classical solution, one may extend the class of admissible minimizers to include Young measures, which are to be interpreted as fine mixtures of functions. The characterization result then implies constraints on the class of relaxed minimizers to be considered. In the final Section~\ref{sc:splitting} we also give another interesting new application to the splitting of generating sequences for generalized gradient Young measure with an \enquote{atomic} part.

This work considers Young measures related to the space BV of functions of bounded variation (see the following section for definitions). Owing to the possible presence of concentration effects, sequences from these spaces are not amenable to a treatment with classical Young measures. A remedy in the form of generalized Young measures was introduced by DiPerna and Majda~\cite{DiPMaj87OCWS} in the context of fluid dynamics. We here follow the Young measure framework of~\cite{KriRin10CGGY,Rind11LSIF,AliBou97NUIG}; for a recent overview and historical remarks see~\cite{Rind11PhD}.

The first characterization theorem for Young measures generated by sequences in BV was presented in~\cite{KriRin10CGGY} and its proof was based on the BV-lower semicontinuity theorem in~\cite{KriRin10RSIF}, which was established without Young measure theory. The lower semicontinuity or \enquote{necessity} part of that characterization theorem was later also proved directly and without the use of Alberti's Rank One Theorem~\cite{Albe93ROPD} in~\cite{Rind12LSYM}.

The aim of this work is two-fold: First, a cleaner proof of the \enquote{sufficiency} part of the BV-Young measure characterization theorem is presented. In particular, the argument in this work is self-contained within the framework of generalized Young measures and does not rely on a relaxation theorem such as~\cite{KriRin10RSIF}. Moreover, it is conceptually much simpler and more satisfying than the previous proof. In the course of the argument, we also introduce the new concept of homogeneous generalized Young measures, which exhibits a number of differences to the concept of classical homogeneous Young measures, as well as novel spaces of \enquote{local} Young measures (separately for regular and singular points), which should prove to be useful in related problems as well. Second, we show how the characterization theorem in BV can be used to prove a natural conjecture about gradient Young measures with an atomic part. 

In all of the following, $\Omega$ is an open bounded Lipschitz domain. The main result in BV is (see the following section for notation):

\begin{theorem} \label{thm:GYM_characterization}
Let $\nu \in \Ybf(\Omega;\R^{m \times d})$ be a (generalized) Young measure with $\lambda_\nu(\partial \Omega) = 0$. Then $\nu$ is a gradient Young measure, $\nu \in \GY(\Omega;\R^m)$, if and only if there exists $u \in \BV(\Omega;\R^m)$ with
\[
  [\nu] := \dpr{\id,\nu_x} \, \Lcal_x^d \restrict \Omega + \dpr{\id,\nu_x^\infty} \, \lambda_\nu(\di x) = Du,
\]
and for all quasiconvex $h \in \Crm(\R^{m \times d})$ with linear growth at infinity the following two Jensen-type inequalities hold:
\[
  \text{(i)}\qquad  h \biggl( \dprb{\id,\nu_x} + \dprb{\id,\nu_x^\infty} \frac{\di \lambda_\nu}{\di \Lcal^d}(x) \biggr)
    \leq \dprb{h,\nu_x} + \dprb{h^\#,\nu_x^\infty} \frac{\di \lambda_\nu}{\di \Lcal^d}(x)
\]
for $\Lcal^d$-almost every $x \in \Omega$, and
\[
  \text{(ii)}\qquad  h^\# \bigl( \dprb{\id,\nu_x^\infty} \bigr) \leq \dprb{h^\#,\nu_x^\infty}
\]
for $\lambda_\nu^s$-almost every $x \in \Omega$.
\end{theorem}

Here, $h^\#$ is the generalized recession function as defined in~\eqref{eq:f_hash}.

We remark that by a recent result of Kirchheim and Kristensen~\cite{KirKri11ACR1} in conjunction with Alberti's Rank One Theorem~\cite{Albe93ROPD}, condition~(ii) is always satisfied, see Remark~\ref{rem:automatic_convexity} for details.

By Theorem~5.4 of~\cite{Rind12LSYM}, the \enquote{necessity} or lower semicontinuity part of the characterization result holds true; in this context note that due to the goal of avoiding Alberti's Rank One Theorem~\cite{Albe93ROPD}, the work~\cite{Rind12LSYM} established the classical lower semicontinuity theorem only under the additional assumption that the \emph{strong} recession function exists (see Remark~5.6 in \emph{loc.\ cit.} for an explanation why this is optimal), but the above Jensen-type inequalities were indeed shown for general quasiconvex functions with linear growth. It remains to show the \enquote{sufficiency} part of the previous theorem, i.e.\ we need to prove:

\begin{proposition} \label{prop:GYM_sufficient}
Let $\nu \in \Ybf(\Omega;\R^{m \times d})$ be a Young measure with $\lambda_\nu(\partial \Omega) = 0$ such that there exists $u \in \BV(\Omega;\R^m)$ with $[\nu] = Du$, and such that for all quasiconvex $h \in \Crm(\R^{m \times d})$ with linear growth at infinity the Jensen-type inequalities~(i) and~(ii) from the preceding theorem hold. Then, $\nu \in \GY(\Omega;\R^{m \times d})$.
\end{proposition}

In contrast to the proof strategy in~\cite{KriRin10CGGY} we here rely on a local approach, i.e.\ we consider blow-ups of Young measures in the form of tangent Young measures as introduced in~\cite{Rind12LSYM} (we use the newer, slightly streamlined versions of the localization principles from~\cite{Rind11PhD}), see the pivotal Propositions~\ref{prop:localize_reg},~\ref{prop:localize_sing} for the precise formulation involving \enquote{tangent Young measures}. This strategy is somewhat reminiscent of the one employed by Kinderlehrer \& Pedregal~\cite{KinPed91CYMG}, which proceeds by reduction to the case of \enquote{homogeneous} Young measures and a geometric Hahn--Banach argument. Owing to the more complex fine structure of BV-functions as opposed to $W^{1,1}$-functions, the present proof of course requires additional ideas. In particular, the local arguments differ at \enquote{regular} and \enquote{singular} points. At the latter points we need to employ Alberti's Rank-One Theorem~\cite{Albe93ROPD} to infer that the local structure of the tangent Young measures is sufficiently constrained (this use of Alberti's theorem seems to be a genuine requirement and Alberti's theorem cannot be replaced by a rigidity lemma as in~\cite{Rind12LSYM}). After the \enquote{local characterization} is complete and we have obtained generating sequences (consisting of gradients) for all tangent Young measures, we glue these sequences together in order to construct a sequence of gradients generating the Young measure we started with. In all of these arguments we rely heavily on the machinery for generalized Young measures developed in~\cite{KriRin10CGGY,Rind12LSYM,Rind11LSIF,Rind11PhD}.

The paper is organized as follows: After recalling basic facts, in particular about generalized Young measures, in Section~\ref{sc:setup}, we start by showing a local version of the BV-Young measure characterization at regular points in Section~\ref{sc:reg} and then at singular points in Section~\ref{sc:sing}. With these local results at hand, we prove Theorem~\ref{thm:GYM_characterization} in Section~\ref{sc:proof}. The final Section~\ref{sc:splitting} closes the paper with the aforementioned application of the characterization result to the splitting of generating sequences.

\section*{Acknowledgements}

The author wishes to thank Jan Kristensen for many stimulating discussions and for reading a preliminary version of the manuscript. Further, thanks are due to the anonymous referees for a careful reading of the manuscript. The support of the Oxford Centre for Nonlinear PDE through the EPSRC Science and Innovation Award to OxPDE (EP/E035027/1) is gratefully acknowledged.

\section{Setup} \label{sc:setup}

We define the balls $B(x,r) := x + r \Bbb^d$ ($r > 0$), where $\Bbb^d$ is the unit ball, and analogously the cubes $Q(x,r) := x + rQ$, where $Q$ is a (possibly rotated) unit cube. The standard unit cube is denoted by $Q(0,1) := (-1/2,1/2)^d$, $\Sbb^{d-1}$ is the unit sphere in $\R^d$ and $\omega_d = \abs{\Bbb^d}$ is the volume of $\Bbb^d$. The matrix space $\R^{m \times d}$ will always be equipped with the Frobenius norm (the Euclidean norm in $\R^{md}$), its unit ball and sphere are $\Bbb^{m \times d}$ and $\partial \Bbb^{m \times d}$, respectively. In all of the following, $\Omega \subset \R^d$ denotes an open bounded Lipschitz domain.

\subsection{BV-functions and quasiconvexity} \label{ssc:BVQC}

A function $u \in \Lrm^1(\Omega;\R^m)$ is said to be a function of bounded variation, $u \in \BV(\Omega;\R^m)$, if its distributional derivative $Du$ can be represented as a finite $\R^{m \times d}$-valued Radon measure carried by $\Omega$. The space $\BV$ is a non-reflexive Banach space under the norm $\norm{u}_{\BV} := \norm{u}_{\Lrm^1} + \abs{Du}(\Omega)$, where $\abs{Du}$ is the total variation measure of $Du$. We will often use the Lebesgue--Radon--Nikod\'{y}m decomposition
\[
  Du = \nabla u \, \Lcal^d \restrict \Omega + D^s u = \nabla u \, \Lcal^d \restrict \Omega + \frac{\di D^s u}{\di \abs{D^s u}} \, \abs{D^s u},
\]
where $\frac{\di D^s u}{\di \abs{D^s u}} \in \Lrm^1(\Omega,\abs{D^s u};\partial \Bbb^{m \times d})$ is the polar function of $D^s u$, i.e.\ the Radon--Nikod\'{y}m derivative of $D^s u$ with respect to its total variation measure $\abs{D^s u}$. Boundedness in norm of a sequence $(u_j) \subset \BV(\Omega;\R^m)$ entails weak* (sequential) compactness, i.e.\ one can find a subsequence (here, like in the following, not relabeled) $u_j \toweakstar u \in \BV(\Omega;\R^m)$, meaning that $u_j \to u$ strongly in $\Lrm^1(\Omega;\R^m)$ and $Du_j \toweakstar Du$ in the sense of Radon measures. For further information on the space BV and its properties we refer to~\cite{AmFuPa00FBVF}, other references are~\cite{Ziem89WDF,EvaGar92MTFP}.

Starting with Morrey's work~\cite{Morr52QSMI}, the natural notion of convexity for minimization problems involving gradients has long been known to be that of quasiconvexity. A locally bounded Borel function $f \colon \R^{m \times d} \to \R$ is said to be \term{quasiconvex} if
\[
  f(A) \leq \dashint_{\Bbb^d} f(A + \nabla \psi(y)) \dd y
  \qquad \text{for all $\psi \in \Crm_c^\infty(\Bbb^d;\R^m)$ and all $A \in \R^{m \times d}$.}
\]
It can be shown that in this definition one may replace $\Bbb^d$ by any bounded open Lipschitz domain and the space $\Crm^\infty$ by $\Wrm^{1,\infty}$ without changing the definition of quasiconvexity. More about this fundamental class of functions can be found in the book~\cite{Daco08DMCV}. We also define the quasiconvex envelope $Qh \colon \R^{m \times d} \to \R \cup \{-\infty\}$ of a continuous function $h  \colon \R^{m \times d} \to \R$ to be the largest quasiconvex function less than or equal to $h$ (possibly identically $-\infty$ if no such function exists). Then, one can show that
\begin{equation} \label{eq:Qh_inf_formula}
  Qh(A) = \inf \, \setBB{ \dashint_{\Bbb^d} h(A + \nabla \psi(y)) \dd y }{ \psi \in \Crm_c^\infty(\Bbb^d;\R^m) },
\end{equation}
which is itself a quasiconvex function or identically $-\infty$, see Section~6.3 in~\cite{Daco08DMCV} and the appendix of~\cite{KinPed91CYMG}.

\subsection{Young measures}

This section gives a brief overview of the basic theory of generalized Young measures and recalls results that will be used later. We follow~\cite{KriRin10CGGY,Rind12LSYM,Rind11LSIF}, also see~\cite{Rind11PhD} for a more comprehensive introduction.

First, we need a suitable class of integrands: Let $\Ebf(\Omega;\R^{m \times d})$ be the set of all $f \in \Crm(\cl{\Omega} \times \R^{m \times d})$ such that
\begin{equation} \label{eq:S_def}
  (Sf)(x,\hat{A}) := (1-\absn{\hat{A}}) \, f \biggl(x, \frac{\hat{A}}{1-\absn{\hat{A}}} \biggr),
  \qquad x \in \cl{\Omega}, \, \hat{A} \in \Bbb^{m \times d}
\end{equation}
extends into a continuous function $Sf \in \Crm(\cl{\Omega \times \Bbb^{m \times d}})$. In particular, this implies that $f$ has \term{linear growth at infinity}, i.e.\ with $M := \norm{f}_{\Ebf(\Omega;\R^{m \times d})} := \norm{Sf \colon \cl{\Omega \times \Bbb^{m \times d}}}_\infty$,
\[
  \abs{f(x,A)} \leq M(1+\abs{A})
  \qquad\text{for all $(x,A) \in \cl{\Omega} \times \R^{m \times d}$.}
\]
For all $f \in \Ebf(\Omega;\R^{m \times d})$ the \term{(strong) recession function}
\[
  f^\infty(x,A) := \lim_{\substack{\!\!\!\! x' \to x \\ \!\!\!\! A' \to A \\ \; t \to \infty}}
    \frac{f(x',tA')}{t},
  \qquad x \in \cl{\Omega},  \,  A \in \R^{m \times d},
\]
exists as a continuous function. Sometimes this notion of a recession function is too strong and so for any function $h \in \Crm(\R^{m \times d})$ with linear growth at infinity we define the \term{generalized recession function}
\begin{equation} \label{eq:f_hash}
  h^\#(A) := \limsup_{\substack{\!\!\!\! A' \to A \\ \; t \to \infty}} \, \frac{h(tA')}{t},  \qquad x \in \cl{\Omega}, \, A \in \R^{m \times d}.
\end{equation}
We remark that for both flavors of recession function one can drop the additional sequence $A' \to A$ if the integrand in question is Lipschitz continuous (which will be the case for all integrands $h$ in this work for which we consider $h^\#$).

\begin{definition} \label{def:gYM}
A \term{(generalized) Young measure} with target space $\R^{m \times d}$ is a triple $\nu = (\nu_x,\lambda_\nu,\nu_x^\infty)$ comprising
\begin{itemize}
  \item[(i)] a parametrized family of probability measures $(\nu_x)_{x \in \Omega} \subset \Mbf^1(\R^{m \times d})$,
  \item[(ii)] a positive finite measure $\lambda_\nu \in \Mbf^+(\cl{\Omega})$ and
  \item[(iii)] a parametrized family of probability measures $(\nu_x^\infty)_{x \in \cl{\Omega}} \subset \Mbf^1(\partial \Bbb^{m \times d})$ (recall that $\partial \Bbb^{m \times d}$ contains all matrices $A \in \R^{m \times d}$ with $\abs{A} = 1$).
\end{itemize}
Moreover, we assume that $\nu$ has the following properties:
\begin{itemize}
  \item[(iv)] the map $x \mapsto \nu_x$ is weakly* measurable with respect to $\Lcal^d$, i.e.\ the function $x \mapsto \dpr{f(x,\frarg),\nu_x}$ is $\Lcal^d$-measurable for all bounded Borel functions $f \colon \Omega \times \R^{m \times d} \to \R$ (here $\dpr{\frarg,\frarg}$ is the Riesz duality pairing between continuous functions and measures),
  \item[(v)] the map $x \mapsto \nu_x^\infty$ is weakly* measurable with respect to $\lambda_\nu$, and
  \item[(vi)] $x \mapsto \dprn{\abs{\frarg},\nu_x} \in \Lrm^1(\Omega)$.
\end{itemize}
We collect all such Young measures $\nu$ in the set $\Ybf(\Omega;\R^{m \times d})$. The parametrized measure $(\nu_x)$ is called the \term{oscillation measure}, the measure $\lambda_\nu$ is the \term{concentration measure}, and $(\nu_x^\infty)$ is the \term{concentration-angle measure}; this terminology is illustrated in~\cite{KriRin10CGGY} and~\cite{Rind11PhD}.
\end{definition}

The \term{duality product} $\ddpr{f,\nu}$ for $f \in \Ebf(\Omega;\R^{m \times d})$ and $\nu \in \Ybf(\Omega;\R^{m \times d})$ is defined via
\begin{align*}
  \ddprb{f,\nu} &:= \int_\Omega \dprb{f(x,\frarg), \nu_x} \dd x
    + \int_{\cl{\Omega}} \dprb{f^\infty(x,\frarg),\nu_x^\infty} \dd \lambda_\nu(x) \\
  &:= \int_\Omega \int_{\R^{m \times d}} f(x,A) \dd \nu_x(A) \dd x
    + \int_{\cl{\Omega}} \int_{\partial \Bbb^{m \times d}} f^\infty(x,A) \dd \nu_x^\infty(A) \dd \lambda_\nu(x).
\end{align*}
One can see easily that $\ddpr{\frarg,\nu}$ for $\nu \in \Ybf(\Omega;\R^{m \times d})$ defines a linear and bounded functional on the Banach space $\Ebf(\Omega;\R^{m \times d})$. Hence, via $\ddpr{\frarg,\frarg}$, a Young measure can be considered a part of the dual space $\Ebf(\Omega;\R^{m \times d})^*$. This embedding gives rise to a weak* topology on $\Ybf(\Omega;\R^{m \times d})$ and so we say that $(\nu_j) \subset \Ybf(\Omega;\R^{m \times d})$ weakly* converges to $\nu \in \Ybf(\Omega;\R^{m \times d})$, in symbols $\nu_j \toweakstar \nu$, if $\ddpr{f,\nu_j} \to \ddpr{f,\nu}$ for all $f \in \Ebf(\Omega;\R^{m \times d})$.

The main compactness result in the space $\Ybf(\Omega;\R^{m \times d})$ states that if $(\nu_j) \subset \Ybf(\Omega;\R^{m \times d})$ is a sequence of Young measures such that
\[
  \supmod_j \, \ddprb{\ONE \otimes \abs{\frarg}, \nu_j} < \infty,
  \qquad\text{where}\qquad
  (\ONE \otimes \abs{\frarg})(x,A) = \abs{A},
\]
then $(\nu_j)$ is weakly* sequentially relatively compact in $\Ybf(\Omega;\R^{m \times d})$, i.e.\ there exists a subsequence (not relabeled) such that $\nu_j \toweakstar \nu$ for some $\nu \in \Ybf(\Omega;\R^{m \times d})$. It can also be shown that the set $\Ybf(\Omega;\R^{m \times d})$ is topologically weakly*-closed (as a subset of $\Ebf(\Omega;\R^{m \times d})^*$).

Another important notion is that of the \term{barycenter}
\[
  [\nu] := \dprb{\id,\nu_x} \, \Lcal_x^d \restrict \Omega + \dprb{\id,\nu_x^\infty} \, \lambda_\nu(\di x)
  \quad \in \Mbf(\cl{\Omega};\R^{m \times d})
\]
of a Young measure $\nu \in \Ybf(\Omega;\R^{m \times d})$; the barycenter is a matrix-valued Radon measure.

The following technical lemma is often useful (a proof of this can be found in the aforementioned references):

\begin{lemma} \label{lem:tensor_products_determine_YM}
There exists a countable set of functions $\{f_k\} = \setn{ \phi_k \otimes h_k }{ k \in \N } \subset \Ebf(\Omega;\R^{m \times d})$, where $\phi_k \in \Crm(\cl{\Omega})$ and $h_k \in \Crm(\R^{m \times d})$, such that $\ddpr{f_k,\nu_1} = \ddpr{f_k,\nu_2}$ for $\nu_1,\nu_2 \in \Ybf(\Omega;\R^{m \times d})$ and all $k \in \N$ implies $\nu_1 = \nu_2$. Moreover, all the $h_k$ can be chosen to be Lipschitz continuous and we may also require that every $h_k$ is either compactly supported in $\R^{m \times d}$ or positively $1$-homogeneous.
\end{lemma}

Next, we define the set $\GY(\Omega;\R^{m \times d})$ of \term{gradient Young measures} as the collection of the Young measures $\nu \in \Ybf(\Omega;\R^{m \times d})$ with the property that there exists a norm-bounded sequence $(u_j) \subset \BV(\Omega;\R^m)$ such that the sequence $(Du_j)$ \term{generates} $\nu$, which in symbols we will write as $Du_j \toY \nu$, meaning that
\[
  \int_\Omega f(x,\nabla u_j(x)) \dd x + \int_\Omega f^\infty \biggl( x, \frac{\di D^s u_j}{\di \abs{D^s u_j}} \biggr) \dd \abs{D^s u_j}(x)  \quad\to\quad \ddprb{f,\nu}
\]
for all $f \in \Ebf(\Omega;\R^{m \times d})$.

The following lemma will be used frequently:

\begin{lemma}\label{lem:boundary_adjustment}
Let $\nu \in \GY(\Omega;\R^{m \times d})$ satisfy $\lambda_\nu(\partial \Omega) = 0$. Then, there exists a generating sequence $(u_j) \subset 
(\Wrm^{1,1} \cap \Crm^\infty)(\Omega;\R^m)$ with $Du_j \toY \nu$ and $u_j|_{\partial \Omega} = u|_{\partial \Omega}$, where $u \in \BV(\Omega;\R^m)$ is an arbitrary underlying deformation of $\nu$, that is, $u$ satisfies $[\nu] \restrict \Omega = Du$.
\end{lemma}

Next, we recall ways to manipulate (gradient) Young measures. These results encapsulate the basic operations that are customarily performed on sequences in the Calculus of Variations, see~\cite{KriRin10CGGY,Rind11PhD} for proofs.

\begin{proposition}[Averaging] \label{prop:averaging}
Let $\nu \in \GY(\Omega;\R^{m \times d})$ satisfy $\lambda_\nu(\partial \Omega) = 0$. Also, assume that $[\nu] = Du$ for a function $u \in \BV(\Omega;\R^m)$ satisfying one of the following two properties:
\begin{itemize}
  \item[(i)] $u$ agrees with an affine function on the boundary $\partial \Omega$ or
  \item[(ii)] $\Omega$ is a cuboid with one face normal $\xi \in \Sbb^{d-1}$ and one (hence every) underlying deformation of $\nu$ is $\xi$-directional.
\end{itemize}
Then, there exists a Young measure $\bar{\nu} \in \GY(\Omega;\R^{m \times d})$ acting on $f \in \Ebf(\Omega;\R^{m \times d})$ as
\begin{equation} \label{eq:averaged_GYM_action}
\begin{aligned}
  \ddprb{f,\bar{\nu}} &= \int_\Omega \dashint_\Omega \dprb{f(x,\frarg),\nu_y} \dd y \dd x \\
  &\qquad + \frac{\lambda_\nu(\cl{\Omega})}{\abs{\Omega}} \int_\Omega \dashint_{\cl{\Omega}}
  \dprb{f^\infty(x,\frarg),\nu_y^\infty} \dd \lambda_\nu(y) \dd x.
\end{aligned}
\end{equation}
More precisely:
\begin{enumerate}
  \item[(1)] The oscillation measure $(\bar{\nu}_x)_x$ is $\Lcal^d$-essentially constant in $x$ and for all $h \in \Crm(\R^{m \times d})$ with linear growth at infinity it holds that
\[
  \qquad\dprb{h,\bar{\nu}_x} = \dashint_\Omega \dprb{h, \nu_y} \dd y  \qquad\text{a.e.}
\]
  \item[(2)] The concentration measure $\lambda_{\bar{\nu}}$ is a multiple of Lebesgue measure, $\lambda_{\bar{\nu}} = \alpha \Lcal^d \restrict \Omega$, where $\alpha = \lambda_\nu(\cl{\Omega}) / \abs{\Omega}$.
  \item[(3)] The concentration-angle measure $(\bar{\nu}_x^\infty)_x$ is $\Lcal^d$-essentially ($\lambda_{\bar{\nu}}$-essentially) constant and for all $h^\infty \in \Crm(\partial \Bbb^{m \times d})$ it holds that
\[
  \qquad\dprb{h^\infty,\bar{\nu}_x^\infty} = \dashint_{\cl{\Omega}} \dprb{h^\infty,\nu_y^\infty}
    \dd \lambda_\nu(y)   \qquad\text{a.e.}
\]
\end{enumerate}
\end{proposition}

Here, the $\xi$-directionality of an underlying deformation $u \in \BV(\Omega;\R^m)$, i.e.\ $[\nu] = Du$, means that $u(x) = v(x \cdot \xi)$ with $v \in \BV(\R;\R^m)$. The proof of case~(i) is contained in Proposition~7 of~\cite{KriRin10CGGY}, the proof of~(ii) is similar, but requires an additional standard staircase construction.

Applying the averaging principle to a fixed function (or, more precisely, an elementary Young measure) yields the following corollary, which we here give in the slightly extended form with a different domain to be covered.

\begin{corollary}[Generalized Riemann--Lebesgue Lemma] \label{cor:Riemann_Lebesgue}
Let $u \in \BV(\Omega;\R^m)$ that agrees with an affine function on $\partial \Omega$. Then, for every open bounded Lipschitz domain $D \subset \R^d$ there exists $\nu \in \GY(D;\R^{m \times d})$ that acts on $f \in \Ebf(D;\R^{m \times d})$ as
\begin{align*}
  \ddprb{f,\nu} &= \int_D \dashint_\Omega f(x,\nabla u(y)) \dd y \dd x \\
  &\qquad + \frac{\abs{D^s u}(\Omega)}{\abs{\Omega}} \int_D \dashint_\Omega f^\infty \biggl(x,
    \frac{\di D^s u}{\di \abs{D^s u}}(y) \biggr) \dd \abs{D^s u}(y) \dd x.
\end{align*}
Moreover, $\lambda_\nu(\partial \Omega) = 0$.
\end{corollary}

We will also need the following approximation result:

\begin{proposition}[Approximation] \label{prop:approximation}
Let $\nu \in \GY(\Omega;\R^{m \times d} )$ satisfy $\lambda_\nu(\partial \Omega) = 0$. Also, assume that $[\nu] = Du$ for a function $u \in \BV(\Omega;\R^m)$ satisfying one of the conditions (i), (ii) from Proposition~\ref{prop:averaging}. Then, for all $k \in \N$, there exists a partition $(C_{kl})_l$ of $(\Lcal^d + \lambda_\nu)$-almost all of $\cl{\Omega}$ into open sets $C_{kl}$, $l = 1,\ldots,N(k)$, with diameters at most $1/k$ and $(\Lcal^d + \lambda_\nu)(\partial C_{kl}) = 0$, and a sequence of Young measures $(\nu_k) \subset \GY(\Omega;\R^{m \times d})$ such that
\[
  \nu_k \toweakstar \nu  \quad\text{in $\Ybf(\Omega;\R^{m \times d})$}
  \qquad\text{as $k \to \infty$}
\]
and for every $f \in \Ebf(\Omega;\R^{m \times d} )$
\begin{align*}
  \ddprb{f,\nu_k} &= \sum_{l = 1}^{N(k)} \ddprb{f,\overline{\nu \restrict {C_{kl}}}} \\
  &= \sum_{l = 1}^{N(k)} \int_{C_{kl}} \dashint_{C_{kl}} \dprb{f(x,\frarg),\nu_y} \dd y \dd x \\
  &\qquad + \sum_{l = 1}^{N(k)} \frac{\lambda_\nu(C_{kl})}{\abs{C_{kl}}} \int_{C_{kl}}
    \dashint_{C_{kl}} \dprb{f^\infty(x,\frarg),\nu_y^\infty} \dd \lambda_\nu(y) \dd x,
\end{align*}
where $\overline{\nu \restrict C_{kl}}$ designates the averaged Young measures to $\nu \restrict C_{kl}$ as in Proposition~\ref{prop:averaging}.
\end{proposition}

For the statements of the following propositions we first introduce \term{local (gradient) Young measures} $\nu \in \mathbf{[G]}\Ybf_\loc(\R^d;\R^{m \times d})$, which are defined on all of $\R^d$ and whose restrictions to any relatively compact set $U \subset \R^d$ lie in $\mathbf{[G]}\Ybf(U;\R^{m \times d})$. We will also use the notion of tangent measures, see~\cite{Rind11PhD} or Chapter~14 of~\cite{Matt95GSME} for this. Suffice it to say here that a \term{tangent measure} $\tau \in \Mbf_\loc(\R^d)$ to a Radon measure $\mu \in \Mbf_\loc(\R^d)$ at the point $x_0 \in \R^d$ is any (local) weak* limit of the pushforward measures $c_n T_*^{(x_0,r_n)} \mu := c_n \mu(x_0 + r_n \frarg)$, where $T^{(x_0,r_n)}(x) := (x-x_0)/r_n$ and $c_n > 0$, $r_n \todown 0$ are sequences of arbitrary (normalization) constants and scaling radii, respectively. Tangent measures are in general not unique and are collected in the set $\Tan(\mu,x_0)$. The works~\cite{Rind11LSIF,Rind12LSYM} extended this concept to Young measures and established the following localization principle:

\begin{proposition}[Localization at regular points] \label{prop:localize_reg}
Let $\nu \in \mathbf{[G]}\Ybf(\Omega;\R^{m \times d})$ be a (gradient) Young measure. Then, for $\Lcal^d$-almost every $x_0 \in \Omega$ there exists a \term{regular (gradient) tangent Young measure} $\sigma \in \mathbf{[G]}\Ybf_\loc(\R^d;\R^{m \times d})$ satisfying
\begin{align*}
  [\sigma] &= \biggl[ \dprb{\id,\nu_{x_0}} + \dprb{\id,\nu_{x_0}^\infty} \frac{\di \lambda_\nu}{\di \Lcal^d}(x_0) \biggr] \, \Lcal^d \in \Tan([\nu],x_0),  &\sigma_y &= \nu_{x_0} \quad\text{a.e.,}  \\
  \lambda_\sigma &= \frac{\di \lambda_\nu}{\di \Lcal^d}(x_0) \, \Lcal^d  \in \Tan(\lambda_\nu,x_0),
    &\sigma_y^\infty &= \nu_{x_0}^\infty \quad\text{a.e.}
\end{align*}
\end{proposition}

To illustrate this result and why it facilitates the term \enquote{tangent Young measure}, we remark that at regular points $x_0$ there exists a sequence $(\sigma^{(r_n)}) \in \mathbf{[G]}\Ybf(\R^d;\R^{m \times d})$, where $r_n \todown 0$, with
\[
  \ddprb{\phi \otimes h, \sigma^{(r_n)}} = \frac{1}{r_n^d} \ddprB{\phi\Bigl(\frac{\frarg-x_0}{r_n}\Bigr) \otimes h,\nu},
  \qquad \phi \otimes h \in \Ebf(\Bbb^d;\R^{m \times d}),
\]
and such that $\sigma^{(r_n)} \toweakstar \sigma$. Thus, $\sigma$ indeed originates from a blow-up construction.

Carrying out a similar procedure for singular points, we observe that the resulting singular tangent Young measures $\sigma \in \Ybf_\loc(\R^d;\R^{m \times d})$ always have the property that $\sigma_y = \delta_0$ almost everywhere (see Section~3.3 of~\cite{Rind11PhD}). This suggests the following definitions:
\begin{align}
  \Ebf^\sing(\Omega;\R^{m \times d}) &:= \setb{ f \in \Ebf(\Omega;\R^{m \times d}) }{ \text{$f(x,\frarg)$ positively $1$-homogeneous, $x \in \Omega$} }, \label{eq:Esing} \\
  \Ybf^\sing(\Omega;\R^{m \times d}) &:= \setb{ \nu = (\nu_x,\lambda_\nu,\nu_x^\infty) \in \Ybf(\Omega;\R^{m \times d}) }{ \text{$\nu_x = \delta_0$ a.e.} }.  \label{eq:Ysing}
\end{align}
The definition of $\Ybf^\sing(\Omega;\R^{m \times d})$ in particular entails that (ii),~(iii),~(v) from Definition~\ref{def:gYM} hold. We call elements of $\Ybf^\sing(\Omega;\R^{m \times d})$ \term{singular Young measures}. Since clearly $\Ybf^\sing(\Omega;\R^{m \times d}) \subset \Ybf(\Omega;\R^{m \times d})$, all results for Young measures also hold for their singular counterparts. For example, the barycenter is
\[
  [\nu] := \dprb{\id,\nu_x^\infty} \, \lambda_\nu(\di x)
  \quad \in \Mbf(\cl{\Omega};\R^{m \times d})
  \qquad \text{for $\nu \in \Ybf^\sing(\Omega;\R^{m \times d})$.}
\]
The duality pairing between $f \in \Ebf^\sing(\Omega;\R^{m \times d})$ and $\nu \in \Ybf^\sing(\Omega;\R^{m \times d})$ is also clear:
\[
  \ddprb{f,\nu} := \int_{\cl{\Omega}} \int_{\partial \Bbb^{m \times d}} f(x,A) \dd \nu_x^\infty(A) \dd \lambda_\nu(x).
\]
We further define the corresponding local spaces $\Ebf^\sing_\loc(\R^d;\R^{m \times d}), \Ybf^\sing_\loc(\R^d;\R^{m \times d})$.

Regarding generation, we now likewise have to restrict attention to positively $1$-homogeneous test functions: For a norm-bounded sequence $(u_j) \subset \BV(\Omega;\R^m)$ we say that $(Du_j)$ \term{generates} $\nu$ \term{(in $\Ybf^\sing$)}, in symbols $Du_j \toY \nu$, if
\[
  \int_\Omega f(x,\nabla u_j(x)) \dd x + \int_\Omega f \biggl( x, \frac{\di D^s u_j}{\di \abs{D^s u_j}} \biggr) \dd \abs{D^s u_j}(x)  \quad\to\quad \ddprb{f,\nu}.
\]
for all $f = f^\infty \in \Ebf^\sing(\Omega;\R^{m \times d})$. We collect all $\nu \in \Ybf^\sing(\Omega;\R^{m \times d})$ with the property that there exists a generating sequence of BV-derivatives in the space $\GY^\sing(\Omega;\R^{m \times d})$ of \term{gradient singular Young measures}.

We can now state a localization principle at singular points:

\begin{proposition}[Localization at singular points] \label{prop:localize_sing}
Let $\nu \in \mathbf{[G]}\Ybf(\Omega;\R^{m \times d})$ be a (gradient) Young measure. Then, for $\lambda_\nu^s$-almost every $x_0 \in \Omega$, there exists a \term{singular (gradient) tangent Young measure} $\sigma \in \mathbf{[G]}\Ybf^\sing_\loc(\R^d;\R^{m \times d})$ satisfying
\begin{equation} \label{eq:loc_sing}
  [\sigma] = \dprb{\id,\nu_{x_0}^\infty} \lambda_\sigma \in \Tan([\nu],x_0), \quad
  \lambda_\sigma \in \Tan(\lambda_\nu^s,x_0) \setminus \{0\}, \quad
  \sigma_y^\infty = \nu_{x_0}^\infty \text{ $\lambda_\sigma$-a.e.}
\end{equation}
\end{proposition}

The proof of this result proceeds via a similar, yet more involved, blow-up principle as in the case of regular points, see~\eqref{eq:sing_tangent_YM}. We omit the proof since the result follows immediately from the singular localization principle in Section~3.3 of~\cite{Rind11PhD} or~\cite{Rind11LSIF}.

\section{Local characterization I: Regular points}  \label{sc:reg}

We first prove a local version of our characterization theorem in BV at regular points, that is for Young measures originating from the regular localization procedure of Proposition~\ref{prop:localize_reg}. This proof is based on the same principles as the one for classical Young measures in~\cite{KinPed91CYMG,KinPed94GYMG}. We still expose it in some detail here for the sake of completeness and also to give an overview of the general strategy, which will also be employed for singular points. Let
\[
  \Ebf^\reg := \setb{ \ONE \otimes h }{ \ONE \otimes h \in \Ebf(\Bbb^d;\R^{m \times d}) },
\]
and for $A_0 \in \R^{m \times d}$,
\begin{align*}
  \Ybf^\reg(A_0) &:= \setb{ \sigma \in \Ybf(\Bbb^d;\R^{m \times d}) }{ \text{$[\sigma] = A_0 \Lcal^d \restrict \Bbb^d$,} \\
  &\qquad\qquad \text{$\sigma_y, \sigma_y^\infty$ constant in $y$, $\lambda_\sigma = \alpha \Lcal^d \restrict \Bbb^d$ for some $\alpha \geq 0$} }, \\
  \GY^\reg(A_0) &:= \Ybf^\reg(A_0) \cap \GY(\Bbb^d;\R^{m \times d}) \\
  &\phantom{:}= \setb{ \sigma \in \Ybf^\reg(A_0) }{ \text{$\exists (v_j) \subset \BV(\Bbb^d;\R^m)$ with $Dv_j \toY \sigma$} }.
\end{align*}
In analogy to the terminology for classical Young measures, we call elements of $\Ybf^\reg(A_0)$ \term{homogeneous} Young measures.

Then, via the usual duality pairing $\ddprn{\frarg,\frarg}$, the space $\Ybf^\reg(A_0)$ can be considered as a part of the dual space to $\Ebf^\reg$. Moreover, the space of test functions $\Ebf^\reg$ is \emph{separating} for $\GY^\reg(A_0)$, that is, for $\sigma_1,\sigma_2 \in \Ybf^\reg(A_0)$,
\[
  \ddprb{f,\sigma_1} = \ddprb{f,\sigma_2}  \quad\text{for all $f \in \Ebf^\reg$}
  \qquad\text{implies}\qquad  \sigma_1 = \sigma_2,
\]
as can be easily checked.

\begin{lemma} \label{lem:GY_reg_properties}
Considered as a subset of the dual space $(\Ebf^\reg)^*$, the set $\GY^\reg(A_0)$ is weakly*-closed and convex.
\end{lemma}

We now prove a local characterization result:

\begin{proof}
Both assertions follow similarly to Lemma~\ref{lem:GY_sing_properties} below (most arguments are in fact easier), once observing that here we even have an affine underlying deformation, which is admissible in the averaging principle, Proposition~\ref{prop:averaging}.
\end{proof}

\begin{proposition} \label{prop:GY_reg_sufficient}
Let $\sigma \in \Ybf^\reg(A_0)$ for some $A_0 \in \R^{m \times d}$. If $\sigma$ satisfies the regular Jensen-type inequality
\begin{equation} \label{eq:charact_cond_reg}
  h(A_0) = h \biggl( \dprb{\id,\sigma_y} + \dprb{\id,\sigma_y^\infty} \, \frac{\di \lambda_\sigma}{\di \Lcal^d}(y) \biggr)
  \leq \dprb{h,\sigma_y} + \dprb{h^\#,\sigma_y^\infty} \, \frac{\di \lambda_\sigma}{\di \Lcal^d}(y)
\end{equation}
for all quasiconvex $h \in \Crm(\R^{m \times d})$ with linear growth at infinity, then $\sigma \in \GY^\reg(A_0)$.
\end{proposition}

\begin{proof}
\proofstep{Step 1.}
From Lemma~\ref{lem:GY_reg_properties} we know that the set $\GY^\reg(A_0)$ is weakly*-closed and convex (considered as a subset of $(\Ebf^\reg)^*$). By the Hahn--Banach Theorem we therefore only need to show that for every weakly*-closed affine half-space $H$ in $(\Ebf^\reg)^*$ with $\GY^\reg(A_0) \subset H$, it holds that $\sigma \in H$. Fix such a half-space $H$ and observe that since $H$ is weakly*-closed, there exists a weakly* continuous functional $G_H \in (\Ebf^\reg)^{**}$ and $\kappa \in \R$ such that
\[
  H = \setb{ e^* \in (\Ebf^\reg)^* }{ G_H(e^*) \geq \kappa }.
\]
From standard arguments in Functional Analysis, see for example Theorem~V.1.3 in~\cite{Conw90CFA}, we may infer that $G_H$ is an evaluation functional, $G_H(e^*) = e^*(f_H)$ for an $f_H \in \Ebf^\reg$. In particular, as $\GY^\reg(A_0) \subset H$, we have
\[
  \ddprb{f_H,\mu} \geq \kappa  \qquad\text{for all $\mu \in \GY^\reg(A_0)$.}
\]
In the remainder of the proof we will show $\ddprn{f_H,\sigma} \geq \kappa$, whereby $\sigma \in H$.

\proofstep{Step 2.}
Let $\eps > 0$. For
\[
  g_\eps(A) := f_H(A) + \eps \abs{A},  \qquad A \in \R^{m \times d},
\]
we immediately see $g_\eps \in \Ebf^\reg$. Next, we observe that $Qf_H(A_0) > - \infty$, where $Qf_H$ is the quasiconvex envelope of $f_H$. Otherwise, there would exist $w \in \Wrm_{A_0 x}^{1,\infty}(\Bbb^d;\R^m)$, that is $w \in \Wrm^{1,\infty}(\Bbb^d;\R^m)$ and $w(x) = A_0 x$ for all $x \in \partial \Bbb^d$, such that
\[
  \int_{\Bbb^d} f_H(\nabla w(y)) \dd y  <  \kappa,
\]
and using the generalized Riemann--Lebesgue Lemma, Corollary~\ref{cor:Riemann_Lebesgue}, we could infer the existence of $\mu \in \GY^\reg(A_0)$ with $\ddprn{f_H,\mu} < \kappa$, a contradiction. Also, by the formula~\eqref{eq:Qh_inf_formula} for the quasiconvex envelope in conjunction with the (classical) Jensen inequality,
\[
  Qg_\eps(A_0) \geq Qf_H(A_0) + \eps \abs{A_0} > -\infty.
\]
Since it is not identically $-\infty$, the function $Qg_\eps$ is quasiconvex, this is proved for example in the appendix of~\cite{KinPed91CYMG}. Next, from $Qg_\epsilon \leq g_\epsilon$ we infer the upper bound $Qg_\epsilon(A) \leq (M+1)(1+\abs{A})$. Because $Qg_\epsilon$ is separately convex and finite by the above reasoning, Lemma~2.5 in~\cite{Kris99LSSW} implies that also $\abs{Qg_\epsilon(A)} \leq \tilde{M}(1+\abs{A})$ for some $\tilde{M} > 0$ that depends on the dimensions $m,d$, and the growth bound $M+1$.

Using $g_\eps \geq Qg_\eps$, $g_\eps^\infty \geq (Qg_\eps)^\#$, and employing the assumption~\eqref{eq:charact_cond_reg} we get
\begin{equation} \label{eq:reg_est_2}
  \ddprb{g_\eps,\sigma} \geq \int_{\Bbb^d} \dprb{Qg_\eps,\sigma_y} + \dprb{(Qg_\eps)^\#,\sigma_y^\infty} \, \frac{\di \lambda_\sigma}{\di \Lcal^d}(y) \dd y \geq Qg_\eps(A_0) \omega_d.
\end{equation}

The formula~\eqref{eq:Qh_inf_formula} for the quasiconvex envelope yields a sequence $(w_j) \subset \Wrm_{A_0 x}^{1,\infty}(\Bbb^d;\R^m)$ with
\[
  \int_{\Bbb^d} g_\eps(\nabla w_j(y)) \dd y  \quad\to\quad  Qg_\eps(A_0) \omega_d.
\]
Moreover, by quasiconvexity of $Qf_H$ and possibly discarding leading elements of the sequence $(w_j)$,
\begin{align*}
  Qg_\eps(A_0) + 1 &\geq \dashint_{\Bbb^d} g_\eps(\nabla w_j(y)) \dd y \geq \dashint_{\Bbb^d} Qf_H(\nabla w_j(y)) + \eps \abs{\nabla w_j(y)} \dd y \\
  &\geq Qf_H(A_0) + \frac{\eps}{\omega_d} \normb{\nabla w_j}_{\Lrm^1(\Bbb^d;\R^{m \times d})}.
\end{align*}
Hence, by the Poincar\'{e}--Friedrichs inequality the sequence $(w_j)$ is uniformly bounded in $\Wrm^{1,1}(\Bbb^d;\R^m)$ and we may assume that $\nabla w_j \toY \mu \in \GY(\Bbb^d;\R^{m \times d})$. Now apply the averaging principle, Proposition~\ref{prop:averaging} to replace $\mu$ by its averaged version $\bar{\mu} \in \GY^\reg(A_0)$. From the properties of $\bar{\mu}$ and the definition of $(w_j)$ we get
\[
  \ddprb{g_\eps,\bar{\mu}} = \ddprb{g_\eps,\mu} = \lim_{j \to \infty} \int_{\Bbb^d} g_\eps(\nabla w_j(y)) \dd y = Qg_\eps(A_0) \omega_d.
\]
Combining with~\eqref{eq:reg_est_2}, we arrive at
\begin{align*}
  \ddprb{f_H,\sigma} &= \ddprb{g_\eps,\sigma} - \eps \ddprb{\ONE \otimes \abs{\frarg},\sigma} \\
  &\geq Qg_\eps(A_0) \omega_d - \eps \ddprb{\ONE \otimes \abs{\frarg},\sigma} \\
  &= \ddprb{g_\eps,\bar{\mu}} - \eps \ddprb{\ONE \otimes \abs{\frarg},\sigma} \\
  &\geq \ddprb{f_H,\bar{\mu}} - \eps \ddprb{\ONE \otimes \abs{\frarg},\sigma} \\
  &\geq \kappa - \eps \ddprb{\ONE \otimes \abs{\frarg},\sigma},
\end{align*}
where the last estimate follows from $\bar{\mu} \in \GY^\reg(A_0) \subset H$. Now let $\eps \todown 0$ to conclude $\ddprb{f_H,\sigma} \geq \kappa$, whence $\sigma \in H$.
\end{proof}

\section{Local characterization II: Singular points}  \label{sc:sing}

Now let $x_0 \in \Omega$ be a singular point of the Young measure $\nu \in \Ybf(\Omega;\R^{m \times d})$ for which we want to show that it is indeed a \emph{gradient} Young measure, and set
\[
  A_0 := \dprb{\id,\nu_{x_0}^\infty}.
\]
We need to distinguish two cases: depending on whether $A_0 = a \otimes \xi$ for $a \in \R^m \setminus \{0\}$ and $\xi \in \Sbb^d$ (i.e.\ $\rank A_0 = 1$), or $A_0 = 0$, we need to work in a slightly different setup. From Alberti's Rank-One Theorem~\cite{Albe93ROPD} we will infer that these two cases are the only ones we need to consider. In the first case every singular tangent Young measure $\sigma$ has a concentration measure $\lambda_\sigma$ that is \term{one-directional in direction $\xi$} (also called \term{$\xi$-directional} in the following), meaning that for any $v \perp \xi$ it holds that $\lambda_\sigma(\frarg + v) = \lambda_\sigma$. In the second case, $\lambda_\sigma$ is arbitrary, but since the underlying deformation has zero boundary values, this case presents no added difficulty.

\subsection{The case $A_0 = a \otimes \xi$.} \label{ssc:sing_rank_one}

Assume that $A_0 = a \otimes \xi$ for $a \in \R^m \setminus \{0\}$ and $\xi \in \Sbb^d$. In all of the following, $Q_\xi$ denotes the (rotated) cube with $\abs{Q_\xi} = 1$ and with one face orthogonal to $\xi \in \Sbb^{d-1}$. Define the space
\begin{align*}
  \Ebf^\sing(\xi) := \setb{ &f \in \Ebf^\sing(Q_\xi;\R^{m \times d}) }{ \text{$f(y,\frarg) = f(y \cdot \xi,\frarg)$} }.
\end{align*}
Here and in the following, $f(y,\frarg) = f(y \cdot \xi,\frarg)$ means that $f$ only depends on $y \cdot \xi$. Notice that $f \in \Ebf^\sing(\xi)$ entails that $f(y,\frarg)$ is positively $1$-homogeneous for all $y \in Q_\xi$, see~\eqref{eq:Esing}. We also set (cf.~\eqref{eq:Ysing})
\begin{align*}
  \Ybf^\sing(a \otimes \xi) &:= \setb{ \sigma \in \Ybf^\sing(Q_\xi;\R^{m \times d}) }{ \text{$[\sigma] = (a \otimes \xi) \, \lambda_\sigma$, $\sigma_y^\infty = \sigma_{y \cdot \xi}^\infty$ and} \\
  &\qquad\qquad\qquad\qquad\qquad\qquad\quad\; \text{$\lambda_\sigma$ is $\xi$-directional} }, \\
  \GY^\sing(a \otimes \xi) &:= \Ybf^\sing(a \otimes \xi) \cap \GY^\sing(Q_\xi;\R^{m \times d}), \\
  \mathbf{[G]}\Ybf_0^\sing(a \otimes \xi) &:=\setb{ \sigma \in \mathbf{[G]}\Ybf^\sing(a \otimes \xi) }{ \lambda_\sigma(\partial Q_\xi) = 0 }.
\end{align*}
Notice that for all $\sigma \in \GY_0^\sing(a \otimes \xi)$ any underlying deformation $u \in \BV(Q_\xi;\R^m)$, i.e.\ with $[\sigma] = Du$, is of the form $u(y) = u_0 + a \psi(y \cdot \xi)$ for some $u_0 \in \R^m$ and $\psi \in \BV(0,1)$. For weak* convergence in $\mathbf{[G]}\Ybf_0^\sing(a \otimes \xi)$ we only require convergence for positively $1$-homogeneous functions $f = f^\infty$. We also observe that
\[
  (a \otimes \xi) \, \lambda_\sigma = [\sigma] = \dprb{\id,\sigma_y^\infty} \, \lambda_\sigma(y),
\]
whereby $\dpr{\id,\sigma_y^\infty} = a \otimes \xi$ is independent of $y$ ($\lambda_\sigma$-a.e.).

Then, via the usual duality pairing $\ddprn{\frarg,\frarg}$, the space $\Ybf^\sing(a \otimes \xi)$ is considered a part of the dual space to $\Ebf^\sing(\xi)$. Again, the space of test functions $\Ebf^\sing(\xi)$ is separating for $\Ybf^\sing(a \otimes \xi)$, that is, if for $\sigma_1,\sigma_2 \in \Ybf^\sing(a \otimes \xi)$ it holds that $\ddpr{f,\sigma_1} = \ddpr{f,\sigma_2}$ for all $f \in \Ebf^\sing(\xi)$, then $\sigma_1 = \sigma_2$. This can be checked by considering elements in $\Ebf^\sing(\xi)$ of the form $f(x,A) := \phi(x)h(A)$ with $\phi \in \Crm(\cl{Q_\xi})$ satisfying $\phi(x) = \phi(x \cdot \xi)$, and $h \in \Crm(\R^{m \times d})$ positively $1$-homogeneous (use the $\xi$-directionality of $y \mapsto \sigma_y^\infty$ and $\lambda_\sigma$).

\begin{lemma} \label{lem:GY_sing_properties}
Considered as subsets of the dual space $\Ebf^\sing(\xi)^*$, the set $\GY^\sing(a \otimes \xi)$ is weakly*-closed and the set $\GY_0^\sing(a \otimes \xi)$ is convex.
\end{lemma}

\begin{proof}
\proofstep{Weak* closedness of $\GY^\sing(a \otimes \xi)$:}
Let $\sigma$ be in the weak* closure of $\GY^\sing(a \otimes \xi)$. Take a countable set $\{f_k\} = \{\phi_k \otimes h_k\} \subset \Ebf^\sing(\xi)$ that determines singular Young measures; this can be constructed by a procedure analogous to the one in Lemma~\ref{lem:tensor_products_determine_YM}. Then, for each $j \in \N$ there exists $\sigma_j \in \GY^\sing(a \otimes \xi)$ such that
\[
  \absb{\ddprb{f_k,\sigma_j} - \ddprb{f_k,\sigma}} \leq \frac{1}{j}  \qquad\text{and}\qquad
  \absb{\ddprb{\ONE \otimes \abs{\frarg},\sigma_j} - 
  \ddprb{\ONE \otimes\abs{\frarg},\sigma}} \leq \frac{1}{j}
\]
whenever $k \leq j$, whereby $\sigma_j \toweakstar \sigma$ in $\GY^\sing(a \otimes \xi)$. Moreover, since $\ONE \otimes \abs{\frarg} \in \Ebf^\sing(\xi)$, a subsequence converges weakly* to a limit in $\Ybf(Q_\xi;\R^{m \times d})$, which must be $\sigma$ because $\Ebf^\sing(\xi) \subset \Ebf(Q_\xi;\R^{m \times d})$. Thus, also $\sigma_j \toweakstar \sigma$ in $\Ybf(Q_\xi;\R^{m \times d})$.  Clearly, $\sigma_y = \delta_0$ for $\Lcal^d$-almost every $y \in Q_\xi$. Since the Young measure convergence implies the convergence of the barycenters (as measures) and $\lambda_{\sigma_j} \toweakstar \lambda_\sigma$ (test with $h(A) := \abs{A}$), we immediately get $[\sigma] = (a \otimes \xi) \lambda_\sigma$ as well.

For all $\phi \otimes h \in \Ebf(Q_\xi;\R^{m \times d})$ with $\supp \phi \subset\subset Q_\xi$ and all sufficiently small $\eta \perp \xi$, we have by the translation-invariance of $y \mapsto (\sigma_j)_y^\infty$ and $\lambda_{\sigma_j}$ with respect to directions orthogonal to $\xi$,
\begin{align*}
  \int_{Q_\xi} \phi(y+\eta) \dprb{h,\sigma_y^\infty} \dd \lambda_\sigma(y)
  &= \lim_{j\to\infty} \int_{Q_\xi} \phi(y+\eta) \dprb{h,(\sigma_j)_y^\infty} \dd \lambda_{\sigma_j}(y) \\
  &= \lim_{j\to\infty} \int_{Q_\xi} \phi(y) \dprb{h,(\sigma_j)_y^\infty} \dd \lambda_{\sigma_j}(y) \\
  &= \int_{Q_\xi} \phi(y) \dprb{h,\sigma_y^\infty} \dd \lambda_\sigma(y).
\end{align*}
Hence, varying $\phi$ and $h$, we may conclude that also $y \mapsto \sigma_y^\infty$ and $\lambda_\sigma$ are $\xi$-directional.

It remains to show that $\sigma$ can be generated by a sequence of gradients. Since all the $\sigma_j$ lie in $\GY^\sing(a \otimes \xi)$, we know that for each $j \in \N$ there exists $u_j \in \Wrm^{1,1}(Q_\xi;\R^m)$ with
\[
  \absBB{\int_Q f_k(x,\nabla u_j(x)) \dd x - \ddprb{f_k,\sigma_j}} \leq \frac{1}{j} 
\]
and
\[
  \absb{\norm{\nabla u_j}_{\Lrm^1(Q_\xi;\R^{m \times d})} - 
  \ddprb{\ONE \otimes\abs{\frarg},\sigma_j}} \leq \frac{1}{j},
\]
again for all $k \leq j$. Thus, the gradients $(\nabla u_j)$ are uniformly bounded in $\Lrm^1(Q_\xi;\R^{m \times d})$ and we may add a constant to every $u_j$ and employ Poincar\'{e}'s inequality to make the sequence $(u_j)$ uniformly bounded in $\Wrm^{1,1}(Q_\xi;\R^m)$. Therefore, there exists a subsequence (not relabeled) with $\nabla u_j \toY \mu \in \GY(Q_\xi;\R^{m \times d})$. The construction entails $\ddpr{f_k,\mu} = \ddpr{f_k,\sigma}$ for all $k \in \N$, hence $\mu = \sigma$.

\proofstep{Convexity of $\GY_0^\sing(a \otimes \xi)$:} Let $\mu, \nu \in \GY^\sing_0(a \otimes \xi)$ and $\theta \in (0,1)$. We need to show that $\theta \mu + (1-\theta)\nu \in \GY^\sing_0(a \otimes \xi)$, the convex combination here being understood in the sense of functionals on $\Ebf^\sing(\xi)$. Without loss of generality we may assume that $\xi = \mathbf{e}_1$, the first unit vector, and $Q_\xi = Q(0,1)$. First we show that it suffices to prove the assertion for averaged measures, where by \enquote{averaged} we mean that they originate from the averaging procedure of Proposition~\ref{prop:averaging}. Indeed, an adaptation of the approximation principle, Proposition~\ref{prop:approximation}, entails that all Young measures $\mu, \nu \in \GY_0^\sing(a \otimes \xi)$ are weak* (sequential) limits of sequences of piecewise homogeneous and averaged gradient Young measures, i.e.\ $\mu = \wslim_{k \to \infty} \mu_k$ and $\nu = \wslim_{k \to \infty} \nu_k$, where with the notation of Proposition~\ref{prop:approximation},
\[
  \ddprb{f,\mu_k} = \sum_{l = 1}^{N(k)} \ddprb{f,\overline{\mu \restrict C_{kl}}} \qquad\text{and}\qquad
  \ddprb{f,\nu_k} = \sum_{l = 1}^{N(k)} \ddprb{f,\overline{\nu \restrict C_{kl}}}.
\]
and $(\Lcal^d + \lambda_\mu + \lambda_\nu)(\partial C_{kl}) = 0$ for all $l = 1,\ldots,N(k)$ and $k \in \N$. From the proof of said proposition in Section~5.3 of~\cite{KriRin10CGGY} we know that for every fixed $l \in \N$ the sets $(C_{kl})_k$ are cuboids arranged in a lattice; in fact it is not difficult to see that we may choose the same open cubes $(C_{kl})$ for both $\mu$ and $\nu$ and such that $\xi = \mathbf{e}_1$ is a face normal to all the $C_{kl}$'s.

We will show below that the set of homogeneous, averaged gradient Young measures is convex. Assuming this for the moment, we get $\theta \overline{\mu \restrict C_{kl}} + (1-\theta) \overline{\nu \restrict C_{kl}} \in \GY_0^\sing(a \otimes \xi)$ for all $\theta \in (0,1)$, where the operation of addition is to be understood in the sense of functionals in $\Ebf^\sing(Q;\R^{m \times d})^*$. Hence, the Young measure
\begin{align*}
  \theta \mu + (1-\theta)\nu &= \wslim_{k \to \infty} \big[\theta \mu_k + (1-\theta) \nu_k\big] \\
  &= \wslim_{k \to \infty} \sum_{l = 1}^{N(k)} \Big[ \theta \overline{\mu \restrict C_{kl}} + (1-\theta) \overline{\nu \restrict C_{kl}} \Big]
\end{align*}
lies in the weak*-closure of $\GY_0^\sing(a \otimes \xi)$. For this note that we can glue together suitable generating sequences of the Young measures $\theta \overline{\mu \restrict C_{kl}} + (1-\theta) \overline{\nu \restrict C_{kl}}$ via a standard staircase construction employing Lemma~\ref{lem:boundary_adjustment} and the one-directionality of the underlying deformations. By the weak* closedness assertion from the first part of the lemma, we may conclude $\theta \mu + (1-\theta)\nu \in \GY_0^\sing(a \otimes \xi)$ ($\lambda_{\theta \mu + (1-\theta)\nu}(\partial Q) = 0$ follows since $\lambda_\mu(\partial Q) = \lambda_\nu(\partial Q) = 0$), proving the second assertion of the present lemma.

To show convexity for the set of homogenous, averaged Young measures used above, let $C \subset Q$ be an open cube with $(\Lcal^d + \lambda_\mu + \lambda_\nu)(\partial C) = 0$. Without loss of generality we further assume $C$ to be the unit cube $Q(0,1) = (-1/2,1/2)^d$. Denote by $\bar{\mu}, \bar{\nu} \in \GY(C;\R^{m \times d})$ the corresponding averagings of $\mu \restrict C$ and $\nu \restrict C$, respectively. In particular (see~\eqref{eq:averaged_GYM_action}),
\begin{align*}
  \ddprb{\phi \otimes h,\bar{\mu}} &= \ddprb{\ONE \otimes h,\mu \restrict C} 
  \dashint_C \phi(x) \dd x, \\
  \ddprb{\phi \otimes h,\bar{\nu}} &= \ddprb{\ONE \otimes h,\nu \restrict C} 
  \dashint_C \phi(x) \dd x
\end{align*}
for all $\phi \in \Crm(C)$, $h \in \Crm(\R^{m \times d})$ with $\phi \otimes h \in \Ebf(C;\R^{m \times d})$. The underlying deformations of $\bar{\mu}, \bar{\nu}$ are affine functions and they only depend on the direction $\xi = \mathbf{e}_1$, so say $[\bar{\mu}] = Du$, $[\bar{\nu}] = Dv$ with $u(x) = q_1 a x_1$ and $v(x) = q_2 a x_1$ for some constants $q_1,q_2 \in \R$.

Let $\nabla u_j \toY \bar{\mu}$ and $\nabla v_j \toY \bar{\nu}$, where $(u_j), (v_j) \subset \Wrm^{1,1}(C;\R^m )$ are both uniformly $\Wrm^{1,1}$-bounded sequences with $u_j \toweakstar u$, $v_j \toweakstar v$. By Lemma~\ref{lem:boundary_adjustment} we can also assume $u_j|_{\partial C} = u|_{\partial C}$, $v_j|_{\partial C} = v|_{\partial C}$ for all $j \in \N$. Define
\[
  D := \setb{ x \in C = Q(0,1) }{ x_1 + 1/2 \leq \theta },
\]
for which it holds that $\abs{D} = \theta \abs{C}$. Then for each $j \in \N$ cover $\Lcal^d$-almost all of $D$ and $C \setminus D$ with (similar) cubes $(a_{jl} + \eps_j C)_l$ and cubes $(b_{jl} + \delta_j C)_l$, respectively, where $\eps_j, \delta_j \leq 1/j$. Define
\[
  w_j(x) := \begin{cases}
      \eps_j u_j \big( \frac{x-a_{jl}}{\eps_j} \big) + u(a_{jl})
        & \text{if } x \in a_{jl} + \eps_j C \; (l \in \N), \\
      \delta_j v_j \big( \frac{x-b_{jl}}{\delta_j} \big) + v(b_{jl}) + \beta a
        & \text{if } x \in b_{jl} + \delta_j C \; (l \in \N).
    \end{cases}
\]
Here, the constant $\beta$ is chosen as to eliminate the jump between the two parts of $C$. The construction of $w_j$ is such that
\[
  \nabla w_j(x) = \begin{cases}
      \nabla u_j \big( \frac{x-a_{jl}}{\eps_j} \big)
        & \text{if } x \in a_{jl} + \eps_j C \; (l \in \N), \\
      \nabla v_j \big( \frac{x-b_{jl}}{\delta_j} \big)
        & \text{if } x \in b_{jl} + \delta_j C \; (l \in \N).
    \end{cases}
\]

It is not difficult to see that $(w_j)$ is still bounded in $\Wrm^{1,1}(C;\R^m)$ and we may assume that $\nabla w_j \toY \gamma \in \GY(C;\R^{m \times d})$. Since effectively in $D$ and $C \setminus D$ we are repeating the construction of the averaging principle, Proposition~\ref{prop:averaging}, we can deduce similarly to the proof of~\eqref{eq:averaged_GYM_action}, see~\cite{KriRin10CGGY}, that for $\phi \in \Crm(\cl{C})$, 
$h \in \Crm(\R^{m \times d})$ positively $1$-homogeneous (in particular $\phi \otimes h \in \Ebf(C;\R^{m \times d})$),
\begin{equation} \label{eq:convex_hom_GYM_action}
\begin{aligned}
  \ddprb{\phi \otimes h,\gamma} &= \ddprb{\ONE_{\cl{C}} \otimes h,\bar{\mu}} \frac{1}{\abs{C}} \int_D \phi(x) \dd x \\
  &\qquad+ \ddprb{\ONE_{\cl{C}} \otimes h,\bar{\nu}} \frac{1}{\abs{C}} \int_{C \setminus D} \phi(x) \dd x.
\end{aligned}
\end{equation}
More precisely, in the step of the proof of the averaging principle where we recognized the Riemann sum, we now have to multiply with $\abs{C}$ to get the correct measure in the sum, hence the division by $\abs{C}$ instead of taking the average in front of the integrals.

Now apply the averaging principle, Proposition~\ref{prop:averaging}, to the measure $\gamma$ to get a gradient Young measure $\bar{\gamma} \in \GY^\sing(C;\R^{m \times d})$ with action
\[
  \ddprb{\phi \otimes h,\bar{\gamma}} = \ddprb{\ONE \otimes h,\gamma} 
  \dashint_C \phi(x) \dd x
\]
for $\phi \in \Crm(\cl{C})$, $h \in \Crm(\R^{m \times d} )$ positively $1$-homogeneous. By~\eqref{eq:convex_hom_GYM_action} (and the extended representation results in~\cite{KriRin10CGGY}),
\begin{align*}
  \ddprb{\phi \otimes h,\bar{\gamma}} &= \Bigg[ \frac{\abs{D}}{\abs{C}} 
  \ddprb{\ONE_{\cl{C}} \otimes h,\bar{\mu}} + \frac{\abs{C \setminus D}}{\abs{C}} 
  \ddprb{\ONE_{\cl{C}} \otimes h,\bar{\nu}} \Bigg] \dashint_C \phi(x) \dd x \\
  &= \theta \ddprb{\phi \otimes h,\bar{\mu}} + (1-\theta) \ddprb{\phi \otimes h,\bar{\nu}}.
\end{align*}
Hence, by Lemma~\ref{lem:tensor_products_determine_YM} we have that $\theta \bar{\mu} + (1-\theta) \bar{\nu} = \bar{\gamma} \in \GY^\sing(C;\R^{m \times d})$ and clearly, this is a homogeneous, averaged Young measure by construction. This concludes the proof. 
\end{proof}

We can now prove a local version of Proposition~\ref{prop:GYM_sufficient} at singular points:

\begin{proposition} \label{prop:GY_sing_rank_1_sufficient}
Let $\sigma \in \Ybf_0^\sing(a \otimes \xi)$ for $a \in \R^m \setminus \{0\}$, $\xi \in \Sbb^d$. If $\sigma$ satisfies the singular Jensen-type inequality
\begin{equation} \label{eq:charact_cond_sing}
  g(a \otimes \xi) = g \bigl( \dprb{\id,\sigma_y^\infty} \bigr) \leq \dprb{g, \sigma_y^\infty}
\end{equation}
for all quasiconvex and positively $1$-homogeneous $g \in \Crm(\R^{m \times d})$ and $\lambda_\sigma$-almost every $y \in Q$, then $\sigma \in \GY_0^\sing(a \otimes \xi)$.
\end{proposition}

\begin{proof}
We employ a similar, yet more involved, Hahn--Banach argument as for regular points. 

\proofstep{Step 1.}
It suffices to show that for every weakly*-closed affine half-space $H$ in $\Ebf^\sing(\xi)^*$ with $\GY_0^\sing(a \otimes \xi) \subset H$, it holds that $\sigma \in H$. By Lemma~\ref{lem:GY_sing_properties} the set $\GY^\sing(a \otimes \xi)$ is weakly*-closed and $\GY_0^\sing(a \otimes \xi)$ is convex, hence the Hahn--Banach Theorem implies that $\sigma$ lies in the weak* closure of $\GY_0^\sing(a \otimes \xi)$, which is contained in $\GY^\sing(a \otimes \xi)$. This then proves the proposition.

For every such half-space $H$ we have analogously to the situation at regular points,
\[
  H = \setb{ e^* \in \Ebf^\sing(\xi)^* }{ e^*(f_H) \geq \kappa }
\]
for some $f_H \in \Ebf^\sing(\xi)$ and $\kappa \in \R$. In particular,
\[
  \ddprb{f_H,\mu} \geq \kappa  \qquad\text{for all $\mu \in \GY_0^\sing(a \otimes \xi)$.}
\]
The goal for the remainder of the proof is to show that $\ddprn{f_H,\sigma} \geq \kappa$.

\proofstep{Step 2.}
For $\eps \in (0,1)$ define
\[
  g_\eps(x,A) := f_H(x,A) + \eps \abs{A},  \qquad x \in Q_\xi, \, A \in \R^{m \times d}.
\]
Clearly, $g_\eps \in \Ebf^\sing(\xi)$. Moreover, fix $\delta > 0$ and take a subdivision of $Q_\xi$ into cuboid slices $S_k$ (with a \enquote{long} face orthogonal to $\xi$), $k = 1,\ldots,N$, satisfying
\[
  \lambda_\sigma(\partial S_k) = 0
\]
and with diameters so small that (see~\eqref{eq:S_def})
\begin{equation} \label{eq:Sgeps_estimate}
  \abs{Sg_\eps(x,A) - Sg_\eps(y,A)} \leq \delta
    \qquad\text{whenever $x,y \in S_k$ ($k \in \N$), $A \in \cl{\Bbb^{m \times d}}$.}
\end{equation}
This is possible by the uniform continuity of $Sg_\eps$ on $\cl{Q_\xi \times \Bbb^{m \times d}}$. 

We claim that in each $S_k$ we can find a point $z_k$ (henceforth fixed) such that the quasiconvex envelope $Qg_\eps(z_k,\frarg)$ of $g_\eps(z_k,\frarg)$ satisfies:
\begin{itemize}
  \item[(A)] $Qg_\eps(z_k,\frarg)$ is finite.
  \item[(B)] $Qg_\eps(z_k,\frarg)$ is quasiconvex and positively $1$-homogeneous (in particular, $Qg_\eps(z_k,\frarg)$ has linear growth at infinity).
  \item[(C)] There exists a sequence $(\psi_j^{(k)})_j \subset \Wrm_{A_0 x}^{1,\infty}(S_k;\R^m)$ with
\[
  \qquad \dashint_{S_k} g_\eps(z_k,\nabla \psi_j^{(k)}(x)) \dd x
  \quad\to\quad
  Qg_\eps(z_k,A_0)
\]
and for some universal constant $c(\eps)$ (not depending on $k,\delta$),
\[
  \qquad \supmod_j \normn{\nabla \psi_j^{(k)}}_{\Lrm^1(S_k;\R^{m \times d} )} \leq c(\eps) \abs{S_k}.
\]
\end{itemize}

To prove (A), it suffices to show that $Qf_H(z_k,A_0) > -\infty$ for at least one $z_k \in S_k$. Indeed, if this condition holds, then also $Qf_H(z_k,A) > -\infty$ for all $A \in \R^{m \times d}$ by the considerations in Section~\ref{ssc:BVQC}. Further, $Qf_H(z_k,\frarg) \in \R$ yields, by the classical Jensen inequality and the Gauss--Green Theorem,
\begin{align*}
  Qf_H(z_k,A) + \eps\abs{A} &\leq \inf \setBB{ \dashint_{\Bbb^d} f_H(z_k,A + \nabla \psi(y)) + \eps\abs{A + \nabla \psi(y)} \dd y }{ \\
  &\qquad\qquad\qquad \psi \in \Crm_c^\infty(\Bbb^d;\R^m) } \\
  &= Qg_\eps(z_k,A)
\end{align*}
for all $A \in \R^{m \times d}$, so in particular $Qg_\eps(z_k,\frarg) > -\infty$.

To show $Qf_H(z_k,A_0) > -\infty$ for at least one $z_k \in S_k$, we prove the stronger assertion that the set of such $z_k$ is even dense in $Q_\xi$. Assume to the contrary that there exists an open slice
\[
  S(z_0,r) := \setb{ x \in Q_\xi }{ \abs{(x-z_0) \cdot \xi} < r }
\]
for $z_0 \in Q_\xi$, $r > 0$, such that $Qf_H(z,A_0) = -\infty$ for all $z \in S(z_0,r)$ (in this context recall that $z \mapsto f_H(z,A_0)$ only depends on $z \cdot \xi$). Then, for every $z \in S(z_0,r)$ we can find $\psi_z \in \Wrm_{A_0 x}^{1,\infty}(Q_\xi,\R^m)$ with
\[
  \dashint_{Q_\xi} f_H(z,\nabla \psi_z(y)) \dd y < \frac{\kappa}{\abs{S(z_0,r)}} - 1.
\]
Moreover, we may assume that the map $z \mapsto \psi_z$ depends on $z \cdot \xi$ only.

For each $z \in S(z_0,r)$ by the uniform continuity of $Sf_H$ choose $\eta(z) = \eta(z \cdot \xi) > 0$ so small that
\begin{equation} \label{eq:Sf_estimate}
  \abs{Sf_H(x,A) - Sf_H(z,A)} \leq \Bigg( \dashint_{Q_\xi} 1 + \abs{\nabla \psi_z(y)} \dd y \Bigg)^{-1}
\end{equation}
for all $x \in S(z,\eta(z))$ and $A \in \Bbb^{m \times d}$. By virtue of Vitali's Covering Theorem, cover $\Lcal^d$-almost all of $S(z_0,r)$ by slices $S_i = S(z_i,r_i) \cap S(z_0,r)$ ($i \in \N$) with $0 < r_i \leq \eta(z_i)$. The generalized Riemann--Lebesgue Lemma, Corollary~\ref{cor:Riemann_Lebesgue}, then yields for each $i$ a gradient Young measure $\mu_i \in \GY(S_k;\R^{m \times d})$ with underlying deformation $A_0 x = a(x \cdot \xi)$ and satisfying
\[
  \ddprb{f_H,\mu_i} = \int_{S_i} \dashint_{Q_\xi} f_H(x,\nabla \psi_{z_i}(y)) \dd y \dd x.
\]
Gluing together all the generating sequences in the slices (again employing a boundary adjustment via Lemma~\ref{lem:boundary_adjustment}), we infer the existence of a gradient Young measure $\mu \in \GY(S(z_0,r);\R^{m \times d})$ acting on $f_H$ as $\ddpr{f_H,\mu} = \sum_{i \in \N} \ddpr{f_H,\mu_i}$. By the choice of the $\eta(z_i)$ and~\eqref{eq:Sf_estimate} we infer
\begin{align*}
  \ddprb{f_H,\mu} &= \sum_{i \in \N} \int_{S_i} \dashint_{Q_\xi} f_H(x,\nabla \psi_{z_i}(y)) \dd y \dd x \\
  &\leq \sum_{i \in \N} \Bigg( \dashint_{Q_\xi} f_H(z_i,\nabla \psi_{z_i}(y)) \dd y + 1 \Bigg)\abs{S_i} < \kappa.
\end{align*}
However, from the above construction it is also clear that $\mu \in \GY_0^\sing(a \otimes \xi)$ (extending $\mu$ by the zero Young measure in $Q_\xi \setminus S(z_0,r)$), hence by assumption $\ddpr{f_H,\mu} \geq \kappa$, a contradiction.

For the proof of (B), we observe that $Qg_\epsilon(z_k,\frarg)$ is quasiconvex, this is analogous to Step~2 in the proof of Proposition~\ref{prop:GY_reg_sufficient}. Moreover, it is not difficult to see from the formula~\eqref{eq:Qh_inf_formula} that positive $1$-homogeneity is preserved when passing to the quasiconvex envelope. We have the upper bound $Qg_\epsilon(z_k,A) \leq g_\epsilon(z_k,A) \leq (M+1)(1+\abs{A})$ and also that $Qg_\epsilon(z_k,\frarg)$ is separately convex and finite. Thus, Lemma~2.5 in~\cite{Kris99LSSW} implies that $\abs{Qg_\epsilon(z_k,A)} \leq \tilde{M}(1+\abs{A})$ for a constant $\tilde{M} > 0$. The decisive point here is that $\tilde{M}$ depends on the dimensions $m,d$, the growth bound $M+1$, and on $Qg_\epsilon(z_k,0) \leq f(z_k,0) \leq M$, but not on $z_k$.

Finally, for assertion (C) we first investigate the coercivity of the functional
\[
  \Gcal_k(\psi) := \dashint_{S_k} g_\epsilon(z_k,\nabla \psi(x)) \dd x,
    \qquad \psi \in \Wrm_{A_0 x}^{1,\infty}(S_k;\R^m).
\]
Since $g_\epsilon(z_k,A) \geq Qg_\epsilon(z_k,A) \geq Qf_H(z_k,A) + \epsilon \abs{A}$, it suffices to check coercivity of the functional with the integrand $Qf_H(z_k,\frarg) + \epsilon \abs{\frarg}$. Let $K \in \R$. By quasiconvexity of $Qf_H(z_k,\frarg)$, the condition
\[
  K \geq \dashint_{S_k} Qf_H(z_k,\nabla \psi(x)) + \epsilon\abs{\nabla \psi(x)} \dd x \geq Qf_H(z_k,A_0)
  + \epsilon \dashint_{S_k} \abs{\nabla \psi(x)} \dd x
\]
for a $\psi \in \Wrm_{A_0 x}^{1,\infty}(S_k;\R^m)$ implies
\[
  \int_{S_k} \abs{\nabla \psi(x)} \dd x \leq \frac{\abs{S_k}}{\epsilon} \Big[ K - Qf_H(z_k,A_0) \Big]
    \leq \frac{\abs{S_k}}{\epsilon} \Big[ K + \tilde{M}(1+\abs{A_0}) \Big].
\]

Now let $(\psi_j ) \subset \Wrm_{A_0 x}^{1,\infty}(S_k;\R^m)$ be a minimizing sequence for $\Gcal_k$. 
Then, the $\psi_j$ satisfy the first assertion in (C) and also, discarding some leading elements in the sequence $(\psi_j)$ if necessary,
\[
  \dashint_{S_k} g_\epsilon(z_k,\nabla \psi_j(x)) \dd x \leq Qg_\epsilon(z_k,A_0) + 1 \leq \tilde{M}(1+\abs{A_0}) + 1.
\]
From the coercivity above we therefore get
\[
  \int_{S_k} \abs{\nabla \psi_j(x)} \dd x \leq \frac{2\tilde{M}(1+\abs{A_0}) + 1}{\epsilon} \abs{S_k}
    =: c(\epsilon) \abs{S_k},
\]
this is the second assertion in (C).

\proofstep{Step 3.}
For all $k = 1,\ldots,N$ pick $z_k \in S_k$ that satisfies the properties (A), (B), (C) above. Then, by virtue of~\eqref{eq:Sgeps_estimate},
\begin{equation} \label{eq:fH_sigma_est_1}
\begin{aligned}
  \ddprb{f_H,\sigma} &= \ddprb{g_\eps,\sigma} - \eps \ddprb{\ONE \otimes \abs{\frarg},\sigma} \\
  &= \sum_{k = 1}^N \int_{S_k} \dprb{g_\eps(y,\frarg),\sigma_y^\infty} \dd \lambda_\sigma(y)
    - \eps \ddprb{\ONE \otimes \abs{\frarg},\sigma}\\
  &\geq \sum_{k = 1}^N \int_{S_k} \dprb{g_\eps(z_k,\frarg),\sigma_y^\infty} \dd \lambda_\sigma(y)
    - (\eps+\delta) \ddprb{\ONE \otimes (1+\abs{\frarg}),\sigma}.
\end{aligned}
\end{equation}
Using $g_\eps(z_k,\frarg) \geq Qg_\eps(z_k,\frarg)$, we infer from the key assumption~\eqref{eq:charact_cond_sing} that
\begin{equation} \label{eq:fH_sigma_est_2}
\begin{aligned}
  \int_{S_k} \dprb{g_\eps(z_k,\frarg),\sigma_y^\infty} \dd \lambda_\sigma(y)
    &\geq \int_{S_k} \dprb{ Qg_\eps(z_k,\frarg), \sigma_y^\infty} \dd \lambda_\sigma(y) \\
  &\geq Qg_\eps(z_k,A_0) \, \lambda_\sigma(S_k).
\end{aligned}
\end{equation}

By assertion (C) above, for each $k = 1,\ldots,N$ there exists a \enquote{recovery} sequence $(\psi_j^{(k)})_j \subset \Wrm_{A_0 x}^{1,\infty}(S_k;\R^m )$ with $\normn{\nabla \psi_j^{(k)}}_{\Lrm^1(S_k;\R^{m \times d} )} \leq c(\eps) \abs{S_k}$ and such that
\[
  \dashint_{S_k} g_\eps(z_k,\nabla \psi_j^{(k)}(x)) \dd x
  \quad\to\quad
  Qg_\eps(z_k,A_0)  \qquad\text{for each $k = 1,\ldots,N$.}
\]
We can now apply the averaging principle as in Proposition~\ref{prop:averaging} to this sequence and take for each $k$ a new sequence $(w_j^{(k)}) \subset \Wrm_{A_0 x}^{1,1}(S_k;\R^m )$, which in addition to the above recovery property (which still holds by a change of variables) satisfies $w_j^{(k)} \toweakstar A_0 x = a(x \cdot \xi)$ as $j \to \infty$. Moreover, the measures $\abs{\nabla w_j^{(k)}} \, \Lcal^d \restrict \Omega $ do not charge the boundary of $S_k$ in the limit. So we may furthermore require $w_j^{(k)}(x)|_{\partial S_k} = A_0 x$.

Recall that $\delta > 0$ was fixed above (and we chose the subdivision of $Q_\xi$ into the slices $S_k$ according to this parameter) and define $w_j^{(\delta)} \colon Q_\xi \to \R^m$ as
\[
  w_j^{(\delta)}(x) := w_j^{(k)}(x) \frac{\lambda_\sigma(S_k)}{\abs{S_k}} + g^{(\delta)}(x \cdot \xi)
  \qquad\text{if $x \in S_k$ ($k = 1,\ldots,N$)},
\]
where $g^{(\delta)}(x \cdot \xi)$ is a staircase term in direction $\xi$, which is chosen precisely to annihilate the jumps otherwise incurred by the difference in $\lambda_\sigma(S_k)$ between adjacent slices. Since $\lambda_\sigma$ is one-directional in direction $\xi$, so is $g^{(\delta)}$. This procedure yields a sequence $(w_j^{(\delta)})_j$, which is uniformly bounded in $\Wrm^{1,1}(Q_\xi;\R^m)$ (for $\delta$ fixed) and which has the property
\begin{equation} \label{eq:fH_sigma_est_3}
  \int_{S_k} g_\eps(z_k,\nabla w_j^{(\delta)}(x)) \dd x
  \quad\to\quad
  Qg_\eps(z_k,A_0) \lambda_\sigma(S_k)  \qquad\text{for every $k = 1,\ldots,N$.}
\end{equation}
The $\Wrm^{1,1}$-uniform boundedness can be seen as follows: We have
\[
  \norm{\nabla w_j^{(\delta)}}_{\Lrm^1(Q_\xi;\R^{m \times d})} \leq c(\eps)\lambda_\sigma(Q_\xi)
\]
for all $j$ and by Poincar\'{e}'s inequality this implies the boundedness of the sequence $(w_j^{(k)})$ in $\Wrm^{1,1}(Q_\xi;\R^m )$.

Combining all the previous arguments, from~\eqref{eq:fH_sigma_est_1},~\eqref{eq:fH_sigma_est_2},~\eqref{eq:fH_sigma_est_3} we conclude
\begin{align}
  \ddprb{f_H,\sigma} &\geq \sum_{k = 1}^N \int_{S_k} \dprb{g_\eps(z_k,\frarg),\sigma_y^\infty} \dd \lambda_\sigma(y)
    - (\eps+\delta) \ddprb{\ONE \otimes (1+\abs{\frarg}),\sigma}  \notag\\
  &\geq \sum_{k = 1}^N Qg_\eps(z_k,A_0) \lambda_\sigma(S_k) - (\eps+\delta) \ddprb{\ONE \otimes (1+\abs{\frarg}),\sigma}  \notag\\
  &= \sum_{k = 1}^N \lim_{j \to \infty} \int_{S_k} g_\eps(z_k,\nabla w_j^{(\delta)}(x)) \dd x
    - (\eps+\delta) \ddprb{\ONE \otimes (1+\abs{\frarg}),\sigma}  \label{eq:fH_sigma_est}
\end{align}
For $\delta > 0$ fixed, separately in every slice $S_k$ apply the averaging principle, Proposition~\ref{prop:averaging}, to the sequence $(\nabla w_j^{(\delta)})$, or, more precisely, to the generated Young measure (restricted to $S_k$). This Young measure exists owing to the uniform boundedness in $\Wrm^{1,1}(Q_\xi;\R^m)$, selecting a subsequence if necessary. We obtain a piecewise homogeneous (on the slices $S_k$) gradient Young measure $\mu^{(\delta)} \in \GY(Q_\xi;\R^{m \times d})$ with $\lambda_\mu(\partial Q_\xi) = 0$ (from the homogeneity) that acts on a function $\ONE \otimes \sum_k h_k \ONE_{S_k} \in \Ebf^\sing(\xi)$ as
\begin{equation} \label{eq:mu_delta}
  \ddprb{\ONE \otimes \sum_k h_k \ONE_{S_k}, \mu^{(\delta)}} = \sum_{k = 1}^N \lim_{j \to \infty} \int_{S_k} h_k(\nabla w_j^{(\delta)}(x)) \dd x.
\end{equation}

\proofstep{Step~4.}
Next, we show how we can achieve that the oscillation measures are equal to $\delta_0$ almost everywhere by creating \enquote{artificial} concentrations.

In all of the following we assume without loss of generality that $\xi = \mathbf{e}_1$ and that the slices $S_k = S(z_k,r_k)$ are arranged in the order $S_1,S_2,\ldots,S_N$ in direction $\xi = \mathbf{e}_1$. For $(w_j^{(\delta)})$ we can assume by construction (and a boundary adjustment) that
\begin{align*}
  w_j^{(\delta)}|_{\partial^{(\pm\mathbf{e}_1)} S_k} &= \mathrm{const} = a q_{j,k}^\pm \in \R^m, \qquad\text{where} \\
  \partial^{(\pm\mathbf{e}_1)} S_k &= \setb{ x \in Q_{\mathbf{e}_1} }{ (x-z_k) \cdot \mathbf{e}_1 = \pm r_k },  \qquad  q_{j,k}^\pm \in \R,
\end{align*}
and $q_{j,k}^+ = q_{j,k+1}^-$ ($k = 1,\ldots,N-1$). Consider $w_j^{(\delta)}$ to be extended periodically on all planes orthogonal to $\xi = \mathbf{e}_1$ and define $v_j^{(\delta)} \in \Wrm^{1,1}(Q_{\mathbf{e}_1};\R^m)$ through
\[
  v_j^{(\delta)}(x) := \begin{cases} 
    w_j^{(\delta)}(z_k + j(x-z_k))      & \text{if $x \in S(z_k, r_k/j)$,} \\
    a q_{j,k}^+ = a q_{j,k+1}^-  & \text{if $x^1-z_k^1 > r_k/j$ and $x^1-z_{k+1}^1 < -r_{k+1}/j$,}
  \end{cases}
\]
for which we get
\[
  \nabla v_j^{(\delta)}(x) = \begin{cases} 
    j \nabla w_j^{(\delta)}(z_k + j(x-z_k))    & \text{if $x \in S(z_k, r_k/j)$,} \\
    0  & \text{if $x^1-z_k^1 > r_k/j$ and $x^1-z_{k+1}^1 < -r_{k+1}/j$.}
  \end{cases}
\]
It is easy to see that the sequence $(v_j^{(\delta)})_j$ is uniformly bounded in $\Wrm^{1,1}(Q_{\mathbf{e}_1};\R^m)$ and that $\nabla v_j^{(\delta)} \to 0$ in measure and almost everywhere as $j\to\infty$. Thus, $(\nabla v_j^{(\delta)})$ generates a Young measure $\nu^{(\delta)} \in \GY(Q_{\mathbf{e}_1};\R^{m \times d})$ with oscillation measure $\nu^{(\delta)}_x = \delta_0$ almost everywhere (cf.~Proposition~2.22 in~\cite{Rind11PhD}).

Now apply once more the averaging principle, Proposition~\ref{prop:averaging}, separately in each $S_k$, to get a gradient Young measure $\overline{\nu}^{(\delta)} \in \GY(Q_{\mathbf{e}_1};\R^{m \times d})$ with
\[
  (\overline{\nu}^{(\delta)})_x = \delta_0 \quad\text{a.e.,}\qquad  
  \text{$\lambda_{\overline{\nu}^{(\delta)}}, (\overline{\nu}^{(\delta)})_x^\infty$ are $\mathbf{e}_1$-directional,}\qquad
  [\overline{\nu}^{(\delta)}] = (a \otimes \xi) \, \Lcal^d \restrict Q_{\mathbf{e}_1},
\]
where the last equality uses the fact that averaging creates \enquote{diffuse} concentrations, see Example~4 in~\cite{KriRin10CGGY}.

We also have $\lambda_{\overline{\nu}^{(\delta)}}(\partial Q_{\mathbf{e}_1}) = 0$, hence
\[
  \overline{\nu}^{(\delta)} \in \GY_0^\sing(a \otimes \xi)
\]
and this Young measure acts on integrands of the form $\ONE \otimes \sum_k h_k \ONE_{S_k} \in \Ebf^\sing(\xi)$ (in particular, all $h_k$ are positively $1$-homogeneous) as
\begin{align*}
  \ddprb{\ONE \otimes \sum_k h_k \ONE_{S_k}, \overline{\nu}^{(\delta)}}
    &= \sum_{k = 1}^N \lim_{j \to \infty} \int_{S(z_k,r_k/j)} h_k \bigl( j \nabla w_j^{(\delta)}(z_k + j(x-z_k)) \bigr) \dd x \\
  &= \sum_{k = 1}^N \lim_{j \to \infty} \int_{S_k} h_k \bigl( \nabla w_j^{(\delta)}(y) \bigr) \dd y \\
  &= \ddprb{\ONE \otimes \sum_k h_k \ONE_{S_k}, \mu^{(\delta)}}.
\end{align*}
Here, to get fromt he first to the second line we employed a change of variables and the periodicity of $\nabla w_j^{(\delta)}$ on planes orthogonal to $\xi = \mathbf{e}_1$.

Combining this with the previous estimates~\eqref{eq:fH_sigma_est},~\eqref{eq:mu_delta}, we arrive at
\begin{align*}
  \ddprb{f_H,\sigma} &\geq \sum_{k = 1}^N \ddprb{g_\eps(z_k,\frarg), \mu^{(\delta)} \restrict S_k } - (\eps+\delta) \ddprb{\ONE \otimes (1+\abs{\frarg}),\sigma} \\
  &=\sum_{k = 1}^N \ddprb{g_\eps(z_k,\frarg), \overline{\nu}^{(\delta)} \restrict S_k } - (\eps+\delta) \ddprb{\ONE \otimes (1+\abs{\frarg}),\sigma} \\
  &\geq \ddprb{g_\eps, \overline{\nu}^{(\delta)}} - \delta \ddprb{\ONE \otimes (1+\abs{\frarg}),\overline{\nu}^{(\delta)}} - (\eps+\delta) \ddprb{\ONE \otimes (1+\abs{\frarg}),\sigma}.
\end{align*}
Now, from the $\delta$-independent $\Wrm^{1,1}$-bound for the functions $w_j^{(\delta)}$ (see above), which propagates to $v_j^{(\delta)}$, we may conclude that the second term vanishes as $\delta \todown 0$. Thus,
\begin{align*}
  \ddprb{f_H,\sigma} &\geq \limsup_{\delta \todown 0} \, \ddprb{g_\eps,\overline{\nu}^{(\delta)}} - \eps \ddprb{\ONE \otimes (1+\abs{\frarg}),\sigma} \\
  &\geq \limsup_{\delta \todown 0} \, \ddprb{f_H,\overline{\nu}^{(\delta)}} - \eps \ddprb{\ONE \otimes (1+\abs{\frarg}),\sigma} \\
  &\geq \kappa - \eps \ddprb{\ONE \otimes \abs{\frarg},\sigma},
\end{align*}
where in the last line we used that $\ddprb{f_H,\overline{\nu}^{(\delta)}} \geq \kappa$ since $\overline{\nu}^{(\delta)} \in \GY_0^\sing(a \otimes \xi) \subset H$. Finally, letting $\eps \todown 0$, we conclude $\ddpr{f_H,\sigma} \geq \kappa$, i.e.\ $\sigma \in H$. Then the Hahn--Banach argument applies and the proof is finished.
\end{proof}

\subsection{The case $A_0 = 0$.} \label{ssc:sing_0}

In this case, we cannot infer anything about the concentration measure (like one-directionality). But since the underlying deformation is zero, the procedures of Section~\ref{ssc:sing_rank_one} still work with slight modifications. Therefore, we can proceed via the same strategy, but working in the following spaces:
\begin{align*}
  \Ebf^\sing(\Sbb^{d-1}) &:= \Ebf^\sing(Q(0,1);\R^{m \times d}),\\
  \Ybf^\sing(0) &:= \setb{ \sigma \in \Ybf^\sing(Q(0,1);\R^{m \times d}) }{ \text{$[\sigma] = 0$} }, \\
  \GY^\sing(0) &:= \Ybf^\sing(0) \cap \GY^\sing(Q(0,1);\R^{m \times d}), \\
  \mathbf{[G]}\Ybf_0^\sing(0) &:=\setb{ \sigma \in \mathbf{[G]}\Ybf^\sing(0) }{ \lambda_\sigma(\partial Q(0,1)) = 0 }.
\end{align*}
The result, for which we omit the proof, is the following:

\begin{proposition} \label{prop:GY_sing_0_sufficient}
Let $\sigma \in \Ybf_0^\sing(0)$. If
\[
  0 = g(0) = g \bigl( \dprb{\id,\sigma_y^\infty} \bigr) \leq \dprb{g, \sigma_y^\infty}
\]
for all quasiconvex and positively $1$-homogeneous $g \in \Crm(\R^{m \times d})$ and $\lambda_\sigma$-almost every $y \in Q(0,1)$, then $\sigma \in \GY_0^\sing(0)$.
\end{proposition}

\section{Proof of the BV characterization theorem} \label{sc:proof}

We are now in a position to prove Proposition~\ref{prop:GYM_sufficient} and thus Theorem~\ref{thm:GYM_characterization}. So let $\nu$ be a Young measure as in the statement of this proposition.

\proofstep{Step 1.}
We first look at regular points. By Proposition~\ref{prop:localize_reg}, at $\Lcal^d$-almost every point $x_0 \in \Omega$, there exists a regular tangent Young measure $\sigma \in \Ybf^\reg(A_0)$ to $\nu$ at $x_0$, where
\[
  A_0 := \dprb{\id,\nu_{x_0}} + \dprb{\id,\nu_{x_0}^\infty} \frac{\di \lambda_\nu}{\di \Lcal^d}(x_0).
\]
Let $h \in \Crm(\R^{m \times d})$ be quasiconvex with linear growth at infinity. By the properties of $\sigma$ (notice in particular that $\frac{\di \lambda_\sigma}{\di \Lcal^d}(y) = \mathrm{const} = \frac{\di \lambda_\nu}{\di \Lcal^d}(x_0)$ a.e.), and assumption~(i) in the statement of Theorem~\ref{thm:GYM_characterization}, we get
\begin{align*}
  h \biggl( \dprb{\id,\sigma_y} + \dprb{\id,\sigma_y^\infty} \frac{\di \lambda_\sigma}{\di \Lcal^d}(y) \biggr)
    &= h \biggl( \dprb{\id,\nu_{x_0}} + \dprb{\id,\nu_{x_0}^\infty} \frac{\di \lambda_\nu}{\di \Lcal^d}(x_0) \biggr) \\
    &\leq \dprb{h,\nu_{x_0}} + \dprb{h^\#,\nu_{x_0}^\infty} \frac{\di \lambda_\nu}{\di \Lcal^d}(x_0) \\
    &= \dprb{h,\sigma_y} + \dprb{h^\#,\sigma_y^\infty} \frac{\di \lambda_\sigma}{\di \Lcal^d}(y)
\end{align*}
for $\Lcal^d$-almost every $y \in \Bbb^d$. Hence,~\eqref{eq:charact_cond_reg} is satisfied, and we may apply Proposition~\ref{prop:GY_reg_sufficient} to infer that $\sigma$ is a regular gradient Young measure, $\sigma \in \GY^\reg(A_0)$.

From a close inspection of the construction of $\sigma$ in the localization principle at regular points, Proposition~\ref{prop:localize_reg} (cf.~\cite{Rind11PhD,Rind11LSIF}), we infer that $\sigma = \wslim_{n\to\infty} \sigma^{(r_n)}$ for a sequence $r_n \todown 0$, where $\sigma^{(r_n)}$ is given by
\[
  \ddprb{\phi \otimes h, \sigma^{(r_n)}} = \frac{1}{r_n^d} \ddprB{\phi\Bigl(\frac{\frarg-x_0}{r_n}\Bigr) \otimes h,\nu},
  \qquad \phi \otimes h \in \Ebf(\Bbb^d;\R^{m \times d}).
\]
In particular, since $\nu$ has underlying deformation $u \in \BV(\Omega;\R^m)$, the Young measure $\sigma^{(r_n)}$ can be assumed to have the underlying deformation
\begin{align*}
  u^{(n)}(y) &= \frac{u(x_0 + r_ny) - \tilde{u}(x_0)}{r_n},  \qquad y \in \Bbb^d,  \qquad\text{so}\\
  Du^{(n)} &= r_n^{-d} T_*^{(x_0,r_n)} Du,
\end{align*}
where $\tilde{u}(x_0)$ is the value of the precise representative of $u$ at $x_0$ and $T_*^{(x_0,r_n)} Du := Du(x_0 + r_n \frarg)$ is the pushforward of $Du$ under the affine transformation $T^{(x_0,r_n)}(x) := (x-x_0)/r_n$. We infer $Du^{(n)} \toweakstar [\sigma] = Dv$ for $v(y) := \nabla u(x_0)y$. Additionally, we even have $u^{(n)} \to v$ with respect to the so-called strict convergence, meaning that $u^{(n)} \to v$ in $\Lrm^1(\Bbb^d)$ and $\abs{Du^{(n)}}(\Bbb^d) \to \abs{Dv}(\Bbb^d)$. This can be seen once again from the proof of the regular localization principle. More precisely, by the assumptions in the proof of said result,
\begin{align*}
  \lim_{k\to\infty} \abs{Du^{(n)}}(\Bbb^d) &= \lim_{k\to\infty} \frac{\abs{Du}(B(x_0,r_n))}{r_n^d} = \omega_d \liminf_{n\to\infty} \frac{\abs{Du}(B(x_0,r_n))}{\abs{B(x_0,r_n)}} \\
  &= \omega_d \abs{\nabla u(x_0)} = \abs{Dv}(\Bbb^d).
\end{align*}
Thus, the strict continuity of the trace operator, see Section~3.8 in~\cite{AmFuPa00FBVF}, implies
\[
  \int_{\partial \Bbb^d} \abs{v - u^{(n)}} \dd \Hcal^{d-1} \to 0.
\]

Let $(v_j) \subset \Wrm^{1,1}(\Bbb^d;\R^m)$ be a norm-bounded generating sequence of $\sigma$ that additionally satisfies $v_j|_{\partial \Bbb^d} = v|_{\partial \Bbb^d}$ (this can always be achieved by Lemma~\ref{lem:boundary_adjustment}). Set
\begin{equation} \label{eq:vjn_def_reg}
  v_j^{(n)}(x) := r_n v_j \Bigl(\frac{x-x_0}{r_n}\Bigr) + \tilde{u}(x_0),  \qquad x \in B(x_0,r_n).
\end{equation}
The above considerations yield by a change of variables,
\begin{align}
  \int_{\partial B(x_0,r_n)} \abs{v_j^{(n)} - u} \dd \Hcal^{d-1} &= r_n^{d-1} \int_{\partial \Bbb^d} \abs{r_n v(y) - u(x_0 + r_ny) + \tilde{u}(x_0)} \dd \Hcal^{d-1}(y) \notag\\
  &= r_n^d \int_{\partial \Bbb^d} \abs{v - u^{(n)}} \dd \Hcal^{d-1} = r_n^d\SmallO(1).
  \label{eq:v_reg_boundary_est}
\end{align}
Here, $\SmallO(1)$ stands for a quantity that vanishes as $n \to \infty$ (or, equivalently, $r_n \todown 0$).

\proofstep{Step 2.}
Turning our attention to singular points, we get from Proposition~\ref{prop:localize_sing} that at $\lambda_\nu^s$-almost every $x_0 \in \Omega$ we can find a singular tangent Young measure $\sigma \in \Ybf^\sing_\loc(\R^d;\R^{m \times d})$ to $\nu$ at $x_0$. Since $[\nu] = Du$, we have
\[
  D^s u = [\nu]^s = \dpr{\id,\nu_x^\infty} \, \lambda_\nu^s(\di x)
\]
and by Alberti's Rank-One Theorem~\cite{Albe93ROPD}, see for example Theorem~3.94 in~\cite{AmFuPa00FBVF}, we get that for $\lambda_\nu^s$-almost every $x_0$,
\[
  \rank\, \dprb{\id,\nu_{x_0}^\infty} \leq 1.
\]
Moreover, if $\rank\, \dprb{\id,\nu_{x_0}^\infty} = 1$, that is,
\[
  A_0 := \dprb{\id,\nu_{x_0}^\infty} = a \otimes \xi,
  \qquad \text{for some $a \in \R^m \setminus \{0\}$, $\xi \in \Sbb^{d-1}$,}
\]
then a \enquote{rigidity} corollary to Alberti's Theorem (Theorem~3.95 in~\cite{AmFuPa00FBVF} or, alternatively, Lemma~3.2 in~\cite{Rind12LSYM}) implies that $\lambda_\sigma^s$ is one-directional in the sense that $\lambda_\sigma^s(A + h) = \lambda_\sigma^s(A)$ for any Borel set $A \subset \R^d$ and $h \perp \xi$. In fact, any underlying deformation $v \in \BV_\loc(\R^d;\R^m)$ of $\sigma$ satisfies $Dv = A_0 \abs{Dv}$, this is a generic property of blow-ups, cf.~Proposition~3.2 in~\cite{Rind11PhD}.

Hence, we have that $\sigma \in \Ybf^\sing(A_0)$ for one of the spaces $\Ybf^\sing(\ldots)$ from either Section~\ref{ssc:sing_rank_one} or~\ref{ssc:sing_0}. Of course, $\sigma \in \Ybf^\sing(A_0)$ really means that the restriction of $\sigma$ to the appropriate cube $Q_\xi(0,1)$ or $Q(0,1)$ lies in $\Ybf^\sing(A_0)$, where $Q_\xi(0,1)$ is the unit cube with one face orthogonal to $\xi \in \Sbb^{d-1}$. Moreover, since we started from a tangent Young measure on all of $\R^d$, we may even assume $\sigma \in \Ybf_0^\sing(A_0)$, i.e.\ $\lambda_\sigma(\partial Q) = 0$, by a simple rescaling argument.

Let $g \in \Crm(\R^{m \times d})$ be quasiconvex and positively $1$-homogeneous. Then, from~\eqref{eq:loc_sing} and assumption~(ii) in Theorem~\ref{thm:GYM_characterization} we get
\begin{align*}
  g \bigl( \dprb{\id,\sigma_y^\infty} \bigr) = g \bigl( \dprb{\id,\nu_{x_0}^\infty} \bigr)
  \leq \dprb{g, \nu_{x_0}^\infty} = \dprb{g, \sigma_y^\infty}
\end{align*}
for $\lambda_\sigma$-almost every $y \in Q$. Therefore, one of the two Propositions~\ref{prop:GY_sing_rank_1_sufficient},~\ref{prop:GY_sing_0_sufficient} (which particular one is applicable depends on $A_0$) implies that $\sigma$ is a gradient Young measure, $\sigma \in \GY_0^\sing(A_0)$. Notice that here the representation of limits through the Young measure holds only for positively $1$-homogeneous $f = f^\infty \in \Ebf^\sing(Q;\R^{m \times d})$, but this will suffice later on.

Again, $\sigma$ is the weak* limit of Young measures $\sigma^{(r_n)}$ for a sequence $r_n \todown 0$, where this time the $\sigma^{(r_n)}$ are given by
\begin{equation} \label{eq:sing_tangent_YM}
  \ddprb{\phi \otimes h, \sigma^{(r_n)}} = c_n \ddprB{\phi\Bigl(\frac{\frarg-x_0}{r_n}\Bigr) \otimes h,\nu},
  \qquad c_n = \frac{1}{\ddprn{\ONE_{Q(x_0,r_n)} \otimes \abs{\frarg},\nu}},
\end{equation}
for all $\phi \otimes h \in \Ebf(Q;\R^{m \times d})$. In particular, for the underlying deformation $u_n \in \BV(Q;\R^m)$ of $\sigma^{(r_n)}$ we may choose
\begin{align*}
  u^{(n)}(y) &= r_n^{d-1}c_n \bigl(u(x_0+r_n y) - \bar{u}^{(n)}\bigr),  \qquad y \in Q,  \qquad\text{so} \\
  Du^{(n)}   &= c_n T_*^{(x_0,r_n)} Du,
\end{align*}
where $\bar{u}^{(n)} := \dashint_{Q(x_0,r_n)} u \dd x$. Moreover, we can assume that $u^{(n)} \to v$ strictly with $[\sigma] = Dv$ for $v \in \BV(Q;\R^m)$. This follows since $\abs{Du^{(n)}} \toweakstar \abs{Dv}$ by properties of blow-ups (see Proposition~3.2 in~\cite{Rind11PhD}) and also $\abs{Dv}(\partial Q) \leq \lambda_\sigma(\partial Q) = 0$, hence $\abs{Du^{(n)}}(Q) \to \abs{Dv}(Q)$ by standard results in measure theory. A detailed proof can be found in Lemma~3.1 of~\cite{Rind12LSYM}.

For a generating sequence $(v_j) \subset \Wrm^{1,1}(Q;\R^m)$ of $\sigma$ with the additional property $v_j|_{\partial Q} = v|_{\partial Q}$, we define
\begin{equation} \label{eq:vjn_def_sing}
  v_j^{(n)}(x) := \frac{1}{r_n^{d-1}c_n} v_j \Bigl(\frac{x-x_0}{r_n}\Bigr) + \bar{u}^{(n)},  \qquad x \in Q(x_0,r_n).
\end{equation}
By the strict continuity of the trace operator, we may therefore derive
\begin{align*}
  \int_{\partial Q(x_0,r_n)} \abs{v_j^{(n)} - u} \dd \Hcal^{d-1} &= r_n^{d-1} \int_{\partial Q} \abs{r_n^{1-d} c_n^{-1} v(y) - u(x_0 + r_ny) + \bar{u}^{(n)}} \dd \Hcal^{d-1}(y) \\
  &= \frac{1}{c_n} \int_{\partial Q} \abs{v - u^{(n)}} \dd \Hcal^{d-1} = c_n^{-1}\SmallO(1).
\end{align*}

\proofstep{Step 3.}
We collect all regular points from Step~1 in the set $R \subset \Omega$ and all singular points from Step~2 in the set $S \subset \Omega$; note that we can assume that both $R$ and $S$ are Borel sets. Then, $(\Lcal^d + \lambda_\nu)(\Omega \setminus (R \cup S)) = 0$ and $R, S$ are disjoint. Further, take a Young measure-determining set of integrands $\{\phi_\ell \otimes h_\ell\}_\ell \subset \Ebf(\Omega;\R^{m \times d})$ as in Lemma~\ref{lem:tensor_products_determine_YM}.

Let $k \in \N$. For each $x \in R$, there exists a regular tangent Young measure $\sigma = \sigma_x \in \GY^\reg(\dpr{\id,\nu_x}+\dpr{\id,\nu_x^\infty}\frac{\di \lambda_\nu}{\di \Lcal^d}(x))$, which is generated by a sequence $(v_j) \subset \Wrm^{1,1}(\Bbb^d;\R^m)$ as in Step~1 (of course, $\sigma$ and $(v_j)$ depend on $x$, but here and in the following we often suppress this dependence for ease of notation if $x$ is fixed and clear from the context); we also consider $v_j^{(n)}$ as defined in~\eqref{eq:vjn_def_reg}. Pick $N = N(x) \in \N$ so large that for $n \geq N$ and all $j \in \N$,
\[
  \int_{\partial B(x,r_n)} \abs{v_j^{(n)} - u} \dd \Hcal^{d-1} \leq \frac{1}{k} \ddprn{\ONE_{B(x,r_n)} \otimes \abs{\frarg},\nu}.
\]
This is possible by~\eqref{eq:v_reg_boundary_est} and the fact that $\ddprn{\ONE_{B(x,r)} \otimes \abs{\frarg},\nu}$ for $x \in R$ asymptotically behaves like $r^d$ as $r \todown 0$ (this, again, can be seen from the proof of the regular localization principle).

Similarly, for all $x \in S$, Step~2 showed the existence of a singular tangent Young measure $\sigma \in \GY_0^\sing(\dpr{\id,\nu_x^\infty}) \subset \Ybf^\sing(Q;\R^{m \times d})$, which is generated by a sequence $(v_j) \subset \Wrm^{1,1}(Q;\R^m)$. Here, $Q$ is either $Q_\xi(0,1)$ if $\dpr{\id,\nu_x^\infty} = a \otimes \xi$ for $a \in \R^m \setminus \{0\}$, $\xi \in \Sbb^{d-1}$, or the usual unit cube $Q(0,1)$ otherwise. For $v_j^{(n)}$ as defined in~\eqref{eq:vjn_def_sing}, there again is $N = N(x) \in \N$ such that for $n \geq N$ we have
\[
  \int_{\partial Q(x,r_n)} \abs{v_j^{(n)} - u} \dd \Hcal^{d-1} \leq \frac{1}{c_n k} = \frac{1}{k} \ddprn{\ONE_{Q(x,r_n)} \otimes \abs{\frarg},\nu}.
\]

The collection of sets $B(x,r_n)$ for $x \in R$, and $Q_{\xi(x)}(x,r_n)$ for $x \in S$ with $r_n \in (0,1/k)$ small enough such that $n \geq N(x)$, is a cover of $\Omega$ that satisfies the assumptions of the Morse Covering Theorem~\cite{Mors47PB} (cited as Theorem~5.51 in~\cite{AmFuPa00FBVF}). Hence, by said theorem, we may find a countable disjoint collection $(K_i(a_i,r_i))_i$ that covers $\Omega$ up to an $(\Lcal^d + \lambda_\nu)$-negligible set, where $K_i(a_i,r_i) := a_i + r_i K_i$ and $K_i$ is either the unit ball $\Bbb^d$ if $a_i \in R$, or an ($a_i$-specific) cube $Q_i = Q_{\xi(a_i)}(a_i,r_i)$ if $a_i \in S$, and $r_i$ is such that $r_i = r_n(a_i)$ for some $n \geq N(a_i)$. Here, $r_n(a_i)$ refers to the $r_n$ associated with the point $a_i$, similarly for $N_i = N(a_i)$, $\sigma_i = \sigma(a_i)$ and $c_i = c_n(a_i)$; in particular, we denote the (regular or singular) tangent Young measure at the point $a_i$ by $\sigma_i$. At singular points $a_i \in S$ we may additionally require that
\[
  \lambda_\nu(\partial Q_i) = 0,  \qquad
  \lambda_{\sigma_i}(\partial Q) = 0
\]
and
\begin{equation} \label{eq:sing_tan_conv}
  \absb{\lambda_{\sigma_i}(Q_i) - (c_i T_*^{(a_i,r_i)}\lambda_\nu^s)(Q_i)} \leq \frac{1}{k}.
\end{equation}
The last condition can be satisfied because $c_n T_*^{(a_i,r_n)}\lambda_\nu^s \toweakstar \lambda_{\sigma_i}$ as $n \to \infty$.

Denote the generating sequence of $\sigma_i$ by $(v_j^{(i)})_j \subset \Wrm^{1,1}(K_i;\R^m)$. Then we choose $j(i,k)$ so large that
\begin{equation} \label{eq:covering_convergence}
  \absBB{ \int_{K_i} h_\ell(\nabla v_{j(i,k)}^{(i)}(x)) \dd x - \ddprb{\ONE_{K_i} \otimes h_\ell,\sigma_i} } \leq \frac{1}{k}  
  \qquad\text{for all $\ell \leq k$.}
\end{equation}
For singular points $a_i \in S$, the preceding requirement only needs to hold for $h_\ell = h_\ell^\infty$ from the subset of $h_\ell$'s that are positively $1$-homogeneous (it holds for such $h_\ell = h_\ell^\infty$ by definition of the weak* convergence in $\GY^\sing(Q;\R^{m \times d})$); in this context recall that all $h_\ell$ either have compact support or are positively $1$-homogeneous.

Define
\[
  w_k(x) := \begin{cases}
              r_i v_{j(i,k)}^{(i)}\bigl(\frac{x-a_i}{r_i}\bigr) + \tilde{u}(a_i) & \text{if $x \in B(a_i,r_i)$, $a_i \in R$,} \\
              r_i^{1-d}c_i^{-1} v_{j(i,k)}^{(i)}\bigl(\frac{x-a_i}{r_i}\bigr) + \bar{u}^{(i)} & \text{if $x \in Q_i(a_i,r_i)$, $a_i \in S$,}
            \end{cases}
\]
where $\tilde{u}(a_i)$ is the value of the precise representative $\tilde{u}$ of $u$ at $a_i$ (which is defined at $a_i \in R$), and $\bar{u}^{(i)} := \dashint_{Q_i(a_i,r_i)} u \dd x$. For the gradient of the $w_k$ we get
\[
  Dw_k = \nabla w_k \, \Lcal^d \restrict \Omega + D^s w_k
\]
with
\[
  \nabla w_k(x) = \begin{cases}
              \nabla v_{j(i,k)}^{(i)}\bigl(\frac{x-a_i}{r_i}\bigr)  & \text{if $x \in B(a_i,r_i)$, $a_i \in R$,} \\
              r_i^{-d}c_i^{-1} \nabla v_{j(i,k)}^{(i)}\bigl(\frac{x-a_i}{r_i}\bigr)  & \text{if $x \in Q_i(a_i,r_i)$, $a_i \in S$.}
            \end{cases}
\]

We can estimate the singular part $D^s w_k$ as follows (by the triangle inequality and the choice of covering):
\begin{align*}
  \abs{D^s w_k}(\Omega) &\leq \sum_i \int_{\partial K_i(a_i,r_i)} \abs{w_k - u} \dd \Hcal^{d-1} 
    \leq \sum_i \frac{1}{k} \ddprn{\ONE_{K_i(a_i,r_i)} \otimes \abs{\frarg},\nu}  \\
  &= \frac{1}{k} \ddprn{\ONE \otimes \abs{\frarg},\nu}.
\end{align*}
As an immediate consequence, for all $f \in \Ebf(\Omega;\R^{m \times d})$ with linear growth constant $M > 0$,
\begin{align*}
  \absBB{ \int_\Omega f^\infty\biggl(x,\frac{\di D^s w_k}{\di \abs{D^s w_k}}(x)\biggr) \dd \abs{D^s w_k}(x)} \leq \frac{M}{k} \ddprb{\ONE \otimes \abs{\frarg},\nu}.
\end{align*}
This estimate implies that to determine the Young measure generated by $Dw_k$ it suffices to consider the absolutely continuous part $\nabla w_k \, \Lcal^d \restrict \Omega$, i.e.\ for $\ell \leq k$ we need to identify the limit as $k \to \infty$ of
\begin{equation} \label{eq:gluing_conv}
\begin{aligned}
  &\int_\Omega \phi_\ell(x) h_\ell(\nabla w_k(x)) \dd x \\
  &\qquad = \sum_{i \colon a_i \in R} \int_{B(a_i,r_i)} \phi_\ell(x) h_\ell \biggl(\nabla v_{j(i,k)}^{(i)}\Bigl(\frac{x-a_i}{r_i}\Bigr)\biggr) \dd x \\
  &\qquad\qquad + \sum_{i \colon a_i \in S} \int_{Q_i(a_i,r_i)} \phi_\ell(x) h_\ell \Bigl(r_i^{-d}c_i^{-1} \nabla v_{j(i,k)}^{(i)}\Bigl(\frac{x-a_i}{r_i}\Bigr)\Bigr) \dd x.
\end{aligned}
\end{equation}

We introduce the following sets (recall that the $a_i$ and $r_i$ depend on the value of $k$, despite this being suppressed in the notation):
\[
  \Rcal_k = \bigcup_{a_i \in R} B(a_i,r_i),  \qquad \Scal_k := \bigcup_{a_i \in S} Q_i(a_i,r_i).
\]

\proofstep{Step 4.}
Assume for now that $h_\ell$ has compact support in $\R^{m \times d}$. Notice that every $a_i \in R$ is a Lebesgue point for the function
\[
  x \mapsto \phi_\ell(x) \biggl( \dprb{h_\ell,\nu_x} + \dprb{h_\ell^\infty,\nu_x^\infty} \frac{\di \lambda_\nu}{\di \Lcal^d}(x) \biggr)
\]
as a consequence of the result on regular localization, Proposition~\ref{prop:localize_reg}. Then, using~\eqref{eq:covering_convergence} for regular points, the integrals in~\eqref{eq:gluing_conv} for $a_i \in R$ satisfy
\begin{align}
  &\int_{B(a_i,r_i)} \phi_\ell(x) h_\ell \biggl(\nabla v_{j(i,k)}^{(i)}\Bigl(\frac{x-a_i}{r_i}\Bigr)\biggr) \dd x  \notag\\
  &\qquad = r_i^d \phi_\ell(a_i) \int_{\Bbb^d} h_\ell(\nabla v_{j(i,k)}^{(i)}(y)) \dd y + \underbrace{r_i^d \SmallO(1) \biggl( 1 + \int_{\Bbb^d} \absb{\nabla v_{j(i,k)}^{(i)}(y)} \dd y \biggr)}_{=: E_i} \notag\\
  &\qquad = r_i^d \phi_\ell(a_i) \ddprb{\ONE_{\Bbb^d} \otimes h_\ell,\sigma_i} + E_i \notag\\
  &\qquad = r_i^d \int_{\Bbb^d} \phi_\ell(a_i) \biggl( \dprb{h_\ell,\nu_{a_i}} + \dprb{h_\ell^\infty,\nu_{a_i}^\infty} \frac{\di \lambda_\nu}{\di \Lcal^d}(a_i) \biggr) \dd y + E_i \notag\\
  &\qquad = r_i^d \int_{\Bbb^d} \phi_\ell(a_i + r_i y) \biggl( \dprb{h_\ell,\nu_{a_i + r_i y}} + \dprb{h_\ell^\infty,\nu_{a_i + r_i y}^\infty} \frac{\di \lambda_\nu}{\di \Lcal^d}(a_i + r_i y) \biggr) \dd y + E_i \notag\\
  &\qquad = \int_{B(a_i,r_i)} \phi_\ell(x) \biggl( \dprb{h_\ell,\nu_x} + \dprb{h_\ell^\infty,\nu_x^\infty} \frac{\di \lambda_\nu}{\di \Lcal^d}(x) \biggr) \dd x + E_i.  \label{eq:main_est_reg}
\end{align}
Here, the $\ell$-dependent term $\SmallO(1)$ goes to zero as $k \to \infty$ (and absorbs all constants, including $\norm{\phi_\ell}_\infty$). Note that the constant in the error term $E_i$ may change from line to line. Changing variables,
\[
 E_i = \SmallO(1) \int_{B(a_i,r_i)} 1 + \absB{\nabla v_{j(i,k)}^{(i)}\Bigl(\frac{x-a_i}{r_i}\Bigr)} \dd x.
\]
In fact, observe that for $\phi_\ell \equiv 1$ and $h_\ell = \abs{\frarg}$ we can apply an analogous reasoning to get
\[
  \int_{B(a_i,r_i)} \absB{\nabla v_{j(i,k)}^{(i)}\Bigl(\frac{x-a_i}{r_i}\Bigr)} \dd x = \ddprb{\ONE_{B(a_i,r_i)} \otimes \abs{\frarg},\nu} + \SmallO(1)\abs{B(a_i,r_i)},
\]
whereby $\norm{\nabla w_k}_{\Lrm^1(\Rcal_k;\R^{m \times d})}$ is uniformly (in $k$) bounded. Plugging~\eqref{eq:main_est_reg} into~\eqref{eq:gluing_conv} and taking into consideration the bound just proved, we have for fixed $\ell$ such that $h_\ell$ has compact support in $\R^{m \times d}$ (whereby it is bounded) 
\begin{align*}
  &\int_\Omega \phi_\ell(x) h_\ell(\nabla w_k(x)) \dd x \\
  &\qquad = \sum_{i \colon a_i \in R} \int_{B(a_i,r_i)} \phi_\ell(x) \biggl( \dprb{h_\ell,\nu_x} + \dprb{h_\ell^\infty,\nu_x^\infty} \frac{\di \lambda_\nu}{\di \Lcal^d}(x) \biggr) \dd x \\
  &\qquad\qquad + \SmallO(1) \bigl[ \abs{\Rcal_k} + \norm{\nabla w_k}_{\Lrm^1(\Rcal_k;\R^{m \times d})} \bigr]
    + \abs{\Scal_k} \cdot \norm{\phi_\ell h_\ell}_\infty \\
  &\qquad = \int_\Omega \phi_\ell(x) \biggl( \dprb{h_\ell,\nu_x} + \dprb{h_\ell^\infty,\nu_x^\infty} \frac{\di \lambda_\nu}{\di \Lcal^d}(x) \biggr) \dd x
    + \SmallO(1)
\end{align*}
since $\abs{\Scal_k} \to 0$ as $k \to \infty$. Thus, we have shown
\begin{equation} \label{eq:main_conv_reg}
  \lim_{k \to \infty} \int_\Omega \phi_\ell(x) h_\ell(\nabla w_k(x)) \dd x = \ddprb{\phi_\ell \otimes h_\ell,\nu}
\end{equation}
for all $\ell \in \N$ such that $h_\ell$ has compact support.

\proofstep{Step 5.}
Next, we recall from Lemma~\ref{lem:tensor_products_determine_YM} that all $h_\ell$ without compact support are in fact positively $1$-homogeneous. Moreover, the proof of the singular localization principle entails that every $a_i \in S$ is a $\lambda_\nu^s$-Lebesgue point of the function
\[
  x \mapsto \dprb{h_\ell^\infty,\nu_x^\infty}.
\]
Thus, for such $h_\ell$ and $a_i \in S$ we get, using~\eqref{eq:sing_tan_conv} and~\eqref{eq:covering_convergence} for singular points as well as the assertions above,
\begin{align}
  &\int_{Q_i(a_i,r_i)} \phi_\ell(x) h_\ell \Bigl(r_i^{-d} c_i^{-1} \nabla v_{j(i,k)}^{(i)}\Bigl(\frac{x-a_i}{r_i}\Bigr)\Bigr) \dd x  \notag\\
  &\qquad = \frac{1}{c_i} \phi_\ell(a_i) \int_{Q_i} h_\ell(\nabla v_{j(i,k)}^{(i)}(y)) \dd y + \underbrace{\frac{\SmallO(1)}{c_i} \biggl( 1 + \int_{Q_i} \absb{\nabla v_{j(i,k)}^{(i)}(y)} \dd y \biggr)}_{=: E_i} \notag\\
  &\qquad = \frac{1}{c_i} \phi_\ell(a_i) \ddprb{\ONE_{Q_i} \otimes h_\ell,\sigma_i} + E_i \notag\\
  &\qquad = \frac{1}{c_i} \int_{Q_i} \phi_\ell(a_i) \dprb{h_\ell,\nu_{a_i}^\infty} \dd \lambda_{\sigma_i}(y) + E_i \notag\\
  &\qquad = \frac{1}{c_i} \int_{Q_i} \phi_\ell(a_i) \dprb{h_\ell,\nu_{a_i}^\infty} \dd (c_i T_*^{(a_i,r_i)}\lambda_\nu^s)(y) + E_i \notag\\
  &\qquad = \int_{Q_i} \phi_\ell(a_i + r_i y) \dprb{h_\ell,\nu_{a_i + r_i y}^\infty} \dd (T_*^{(a_i,r_i)}\lambda_\nu^s)(y) + E_i \notag\\
  &\qquad = \int_{Q_i(a_i,r_i)} \phi_\ell(x) \dprb{h_\ell,\nu_x^\infty} \dd \lambda_\nu^s(x) + E_i, \label{eq:main_est_sing}
\end{align}
where again the constant in $E_i$ may change from line to line. We further get by a change of variables
\[
  E_i = \SmallO(1) \biggl( \ddprn{\ONE_{Q(a_i,r_i)} \otimes \abs{\frarg},\nu} + \int_{Q_i(a_i,r_i)} \absB{r_i^{-d} c_i^{-1} \nabla v_{j(i,k)}^{(i)}\Bigl(\frac{x-a_i}{r_i}\Bigr)} \dd x \biggr),
\]
Also, we may derive in a similar fashion,
\[
  \int_{Q_i(a_i,r_i)} \absB{r_i^{-d} c_i^{-1} \nabla v_{j(i,k)}^{(i)}\Bigl(\frac{x-a_i}{r_i}\Bigr)} \dd x \leq (1+\SmallO(1))\ddprb{\ONE_{Q_i(a_i,r_i)} \otimes \abs{\frarg},\nu},
\]
hence $\norm{\nabla w_k}_{\Lrm^1(\Scal_k;\R^{m \times d})}$ is uniformly bounded. Then, also using the previous step for the regular points,~\eqref{eq:gluing_conv} in conjunction with~\eqref{eq:main_est_sing} yields for fixed $\ell$,
\begin{align*}
  &\int_\Omega \phi_\ell(x) h_\ell(\nabla w_k(x)) \dd x \\
  &\qquad = \sum_{i \colon a_i \in R} \int_{B(a_i,r_i)} \phi_\ell(x) \biggl( \dprb{h_\ell,\nu_x} + \dprb{h_\ell,\nu_x^\infty} \frac{\di \lambda_\nu}{\di \Lcal^d}(x) \biggr) \dd x  \\
  &\qquad\qquad + \sum_{i \colon a_i \in S} \int_{Q_i(a_i,r_i)} \phi_\ell(x) \dprb{h_\ell,\nu_x^\infty} \dd \lambda_\nu^s(x) \\
  &\qquad\qquad + \SmallO(1) \bigl[ \abs{\Rcal_k} + \ddprn{\ONE_{\Scal_k} \otimes \abs{\frarg},\nu} + \norm{\nabla w_k}_{\Lrm^1(\Omega;\R^{m \times d})} \bigr] \\
  &\qquad = \ddprb{\phi_\ell \otimes h_\ell,\nu} + \SmallO(1).
\end{align*}
Thus, we have arrived at
\begin{equation} \label{eq:main_conv_sing}
  \lim_{k \to \infty} \int_\Omega \phi_\ell(x) h_\ell(\nabla w_k(x)) \dd x = \ddprb{\phi_\ell \otimes h_\ell,\nu}
\end{equation}
for all $\ell \in \N$ such that $h_\ell$ is positively $1$-homogeneous.

Since in the course of the above proof we showed a uniform $\Lrm^1$-norm bound on $(\nabla w_k)$, passing to a subsequence if necessary, $(Dw_k)$ generates a gradient Young measure $\mu \in \GY(\Omega;\R^{m \times n})$. But by~\eqref{eq:main_conv_reg} and~\eqref{eq:main_conv_sing}, only $\mu = \nu$ is possible and so $\nu$ has been shown to be a gradient Young measure. \qed

We close this section by the following curious fact, first observed in~\cite{KirKri11ACR1}:

\begin{remark} \label{rem:automatic_convexity}
Condition~(ii) in Theorem~\ref{thm:GYM_characterization} is always satisfied: Since $[\nu] = Du$, it holds that $\dpr{\id,\nu_x^\infty} = \frac{\di D^s u}{\di \abs{D^s u}}(x)$ for $\abs{D^s u}$-almost every $x \in \Omega$ and by Alberti's Rank One Theorem~\cite{Albe93ROPD} this matrix has rank one $\abs{D^s u}$-a.e. The main result of~\cite{KirKri11ACR1} entails that every quasiconvex (hence rank-one convex) and positively $1$-homogeneous $g \colon \R^{m \times d} \to \R$ is in fact convex at the rank-one matrix $\dpr{\id,\nu_x^\infty}$, whereby the second condition in Theorem~\ref{thm:GYM_characterization} reduces to the classical Jensen inequality and so is always satisfied. More precisely, the result from~\cite{KirKri11ACR1} says that there exists a linear function $\ell \colon \R^{m \times d} \to \R$ such that $g(\dpr{\id,\nu_x^\infty}) = \ell(\dpr{\id,\nu_x^\infty})$ and $\ell \leq g$. Then,
\[
  g\bigl(\dprb{\id,\nu_x^\infty}\bigr) = \ell\bigl(\dprb{\id,\nu_x^\infty}\bigr) = \int \ell(A) \dd \nu_x^\infty(A) \leq \int g(A) \dd \nu_x^\infty(A) = \dprb{g,\nu_x^\infty},
\]
which is nothing else than condition~(ii) in Theorem~\ref{thm:GYM_characterization}.
\end{remark}

\section{Splitting of generating sequences}  \label{sc:splitting}

In this final section we briefly discuss an interesting application of the BV-characterization theorem to the splitting of generating sequences for Young measures with an \enquote{atomic} part. Consider a gradient Young measure $\nu \in \GY(\Omega;\R^m)$ with the property that for a given function $v \in \BV(\Omega;\R^m)$ the concentration part of $\nu$ can be split into two mutually singular parts as
\[
  \nu_x^\infty \, \lambda_\nu(\di x) = \nu_x^\infty \, \bigl[ \lambda_\nu - \abs{D^s v} \bigr](\di x) + \delta_{p(x)} \, \abs{D^s v}(\di x),
  \qquad\text{where $p(x) = \displaystyle\frac{\di D^s v}{\di \abs{D^s v}}(x)$.}
\]
More precisely, we assume:
\begin{itemize}
  \item[(i)] The measures $\lambda_\nu - \abs{D^s v}$ and $\abs{D^s v}$ are mutually singular (hence both positive),
  \item[(ii)] $\nu_x^\infty = \delta_{p(x)}$ for $\abs{D^s v}$-almost every $x \in \Omega$.
\end{itemize}
Intuitively, (i) and (ii) mean that $Dv$ is an \enquote{atomic} part of $\nu$ (the absolutely continuous part is actually uncritical, the issue is the concentration part). For reasons of simplicity let us also assume that
\begin{itemize}
  \item[(iii)] $\lambda_\nu(\partial \Omega) = 0$,
\end{itemize}
otherwise one has to embed the functions and Young measures into a larger domain and take into account boundary terms.

The natural conjecture now is that under the assumptions (i)--(iii) one can find a generating sequence $(u_j) \subset \BV(\Omega;\R^m)$ for $\nu$ of the form $u_j = w_j + v$ with $(w_j) \subset \BV(\Omega;\R^m)$ and $Dw_j \toweakstar [\nu] - Dv$, that is, we can cleave the atomic part from the generating sequence. However, when trying to prove this result, one faces the difficulty that it is easy to \emph{add} concentrations, but very difficult in general to \emph{remove} them (cf.\ Proposition~6 in~\cite{KriRin10CGGY} on shifting of Young measures). Indeed, naively setting $w_j := u_j - v$, where $Du_j$ is a generating sequence for $\nu$ will not have the desired effect if $(u_j) \subset \Crm^\infty(\Omega;\R^m)$. In fact, this procedure corresponds to shifting $\nu$ by $-Dv$ via Proposition~6 in~\cite{KriRin10CGGY} and only results in a (gradient) Young measure with the concentration part
\[
  \nu_x^\infty \, \bigl[ \lambda_\nu - \abs{D^s v} \bigr](\di x) + \delta_{p(x)} \, \abs{D^s v}(\di x) + \delta_{-p(x)} \, \abs{D^s v}(\di x),
\]
but the last two parts do not cancel. So, using direct manipulations only, it seems rather hard to cut off the atomic part from a generating sequences, while at the same time preserving the curl-freeness.

We will now use the characterization theorem in BV to prove the conjecture. Define the Young measure $\mu \in \Ybf(\Omega;\R^m)$ as follows:
\begin{align*}
  \dprb{h,\mu_x} &:= \dprb{h(\frarg - \nabla v(x),\nu_x}
    \qquad\text{for $\Lcal^d$-a.e.\ $x \in \Omega$ and all $h \in \Crm_c(\R^{m \times d})$,} \\
  \lambda_\mu &:= \lambda_\nu - \abs{D^s v} \quad  \in \Mbf^+(\cl{\Omega}), \\
  \mu_x^\infty &:= \nu_x^\infty  \qquad\text{for $\lambda_\mu$-a.e.\ $x \in \Omega$.}
\end{align*}
By our assumptions~(i)--(iii) this definition always yields a Young measure $\mu \in \Ybf(\Omega;\R^{m \times d})$ with $[\mu] = [\nu] - Dv$. In particular, (i) entails that $\lambda_\mu$ is a \emph{positive} measure.

Now, if $h \in \Crm(\R^{m \times d})$ is a quasiconvex function with linear growth at infinity, then also $\tilde{h} \colon \R^{m \times d} \to \R$ given by $\tilde{h}(A) := h(A-\nabla v(x))$ for a.e.\ $x \in \Omega$ has these properties, and from the necessity part of the characterization theorem applied to the gradient Young measure $\nu$ it follows for $\Lcal^d$-almost every $x \in \Omega$ that
\begin{align*}
  h \biggl( \dprb{\id,\mu_x} + \dprb{\id,\mu_x^\infty} \frac{\di \lambda_\mu}{\di \Lcal^d}(x) \biggr)
    &= \tilde{h} \biggl( \dprb{\id,\nu_x} + \dprb{\id,\nu_x^\infty} \frac{\di \lambda_\nu}{\di \Lcal^d}(x) \biggr) \\
  &\leq \dprb{\tilde{h},\nu_x} + \dprb{\tilde{h}^\#,\nu_x^\infty} \frac{\di \lambda_\nu}{\di \Lcal^d}(x) \\
  &= \dprb{h,\mu_x} + \dprb{h^\#,\mu_x^\infty} \frac{\di \lambda_\mu}{\di \Lcal^d}(x)
\end{align*}
since $\tilde{h}^\# = h^\#$ and $\mu_x^\infty = \nu_x^\infty$ $\Lcal^d$-almost everywhere. Moreover,
\[
  h^\# \bigl( \dprb{\id,\mu_x^\infty} \bigr) = h^\# \bigl( \dprb{\id,\nu_x^\infty} \bigr)
  \leq \dprb{h^\#,\nu_x^\infty} = \dprb{h^\#,\mu_x^\infty}
\]
for $\lambda_\mu^s$-almost every $x \in \Omega$.

The previous calculations show that $\mu$ satisfies the assumptions of the sufficiency part of the characterization theorem, whereby $\mu \in \GY(\Omega;\R^m)$. Hence, there exists a sequence $(w_j) \subset (\Crm^\infty \cap \Wrm^{1,1})(\Omega;\R^m)$ (also see Lemma~\ref{lem:boundary_adjustment}) with $Dw_j \toY \mu$. Then it is not difficult to see that the sequence $Dw_j + Dv$ indeed generates the original Young measure $\nu$, for details see Proposition~6 in~\cite{KriRin10CGGY}. Hence we have proved:

\begin{theorem}
Let $\nu \in \GY(\Omega;\R^m)$ satisfy the conditions~(i)--(iii) above. Then, there exists a sequence $(u_j) \subset \BV(\Omega;\R^m)$ for $\nu$ that can be split as
\[
  u_j = w_j + v  \qquad\text{with $(w_j) \subset (\Crm^\infty \cap \Wrm^{1,1})(\Omega;\R^m)$ and $Dw_j \toweakstar [\nu] - Dv$.}
\]
and such that $Du_j \toY \nu$.
\end{theorem}

We conclude with remarks about the situation where $\lambda_\nu - \abs{D^s v}$ and $\abs{D^s v}$ are not mutually singular and the concentration effects may interfere (for example in the common situation that $v = u$, where $[\nu] = Du$). In this case we need to additionally \emph{require} that the concentration part of $\nu$ can be written in the form
\[
  \nu_x^\infty \, \lambda_\nu(\di x) = \mu_x^\infty \, \lambda_\mu(\di x) + \delta_{p(x)} \, \abs{D^s v}(\di x),
  \qquad p(x) = \frac{\di D^s v}{\di \abs{D^s v}}(x),
\]
for another family $(\mu_x^\infty)_{x \in \Omega} \subset \Mbf^+(\partial \Bbb^{m \times d})$ of \emph{positive} (sub-probability) measures and a \emph{positive} measure $\lambda_\mu \in \Mbf^+(\cl{\Omega})$ (notice that these positivity properties are not automatic anymore). Moreover, absorbing an appropriate factor into $\lambda_\mu$, it can always be assumed that $\mu^\infty_x$ is a probability measure. The regular Jensen-type inequality is proved exactly as before and by the main result of~\cite{KirKri11ACR1}, the singular Jensen-type inequality is always satisfied, cf.\ Remark~\ref{rem:automatic_convexity}, so we can again employ the characterization result to conclude.

%
%


\providecommand{\bysame}{\leavevmode\hbox to3em{\hrulefill}\thinspace}
\providecommand{\MR}{\relax\ifhmode\unskip\space\fi MR }
\providecommand{\MRhref}[2]{%
  \href{http://www.ams.org/mathscinet-getitem?mr=#1}{#2}
}
\providecommand{\href}[2]{#2}

\end{document}